\documentclass[onefignum,onetabnum]{siamart190516}

\usepackage{booktabs}
\usepackage{wrapfig}

\newcommand{\bH}{\boldsymbol{\mathcal{H}}}
\newcommand{\bD}{\boldsymbol{\mathcal{D}}}

\newcommand{\Image}{\operatorname{Im}}
\renewcommand{\ker}{\operatorname{Ker}}
\newcommand{\Fix}{\operatorname{Fix}}
\newcommand{\dom}{\operatorname{dom}}
\newcommand{\zer}{\operatorname{zer}}

\usepackage{lipsum}
\usepackage{amsfonts}
\usepackage{graphicx}
\usepackage{epstopdf}
\usepackage{algorithmic}
\ifpdf
  \DeclareGraphicsExtensions{.eps,.pdf,.png,.jpg}
\else
  \DeclareGraphicsExtensions{.eps}
\fi

\newsiamremark{remark}{Remark}
\newsiamremark{example}{Example}
\newsiamremark{hypothesis}{Hypothesis}
\crefname{hypothesis}{Hypothesis}{Hypotheses}
\newsiamthm{claim}{Claim}
\newsiamthm{prop}{Proposition}
\newsiamthm{defi}{Definition}

\headers{Degenerate Preconditioned Proximal Point Algorithms}{K. Bredies, E. Chenchene, D. Lorenz, E. Naldi}

\title{Degenerate Preconditioned Proximal Point Algorithms\thanks{%
    \today %Submitted to the editors DATE.
\funding{This work has received funding from the European Union’s Framework Programme for Research and Innovation Horizon 2020 (2014-2020) under the Marie Skłodowska-Curie Grant Agreement No. 861137. \raisebox{-0.25ex}{\includegraphics[height=2ex]{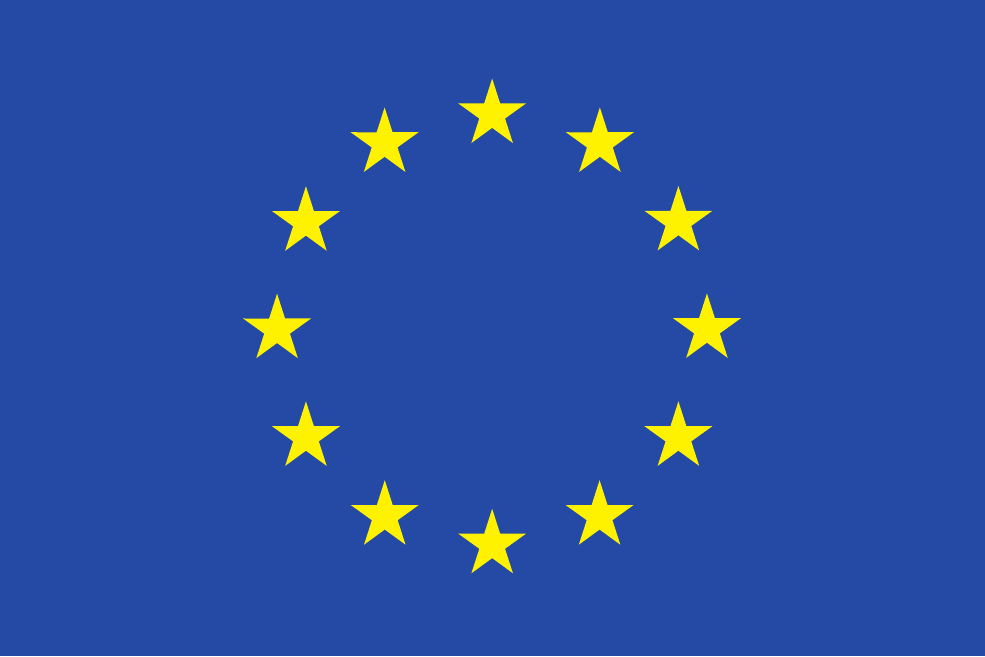}}}\\The Institute of Mathematics and Scientific Computing at the University of Graz, with which K.B.~and E.C.~are affiliated, is a member of NAWI Graz (\texttt{https://nawigraz.at/en}).}}

\author{Kristian Bredies\thanks{Institute of Mathematics and Scientific Computing, University of Graz, Graz, Austria, email:kristian.bredies@uni-graz.at, enis.chenchene@uni-graz.at}
\and Enis Chenchene\footnotemark[2]
\and Dirk Lorenz\thanks{Institute of Analysis and Algebra, TU Braunschweig, email: d.lorenz@tu-braunschweig.de, e.naldi@tu-braunschweig.de}
\and Emanuele Naldi\footnotemark[3]}

\usepackage{amsopn}

\ifpdf
\hypersetup{
  pdftitle={Degenerate Preconditioned Proximal Point Algorithms},
  pdfauthor={K. Bredies, E. Chenchene, D. Lorenz, E. Naldi}
}
\fi

\begin{document}

\maketitle

\begin{abstract}
  In this paper we describe a systematic procedure to analyze the convergence of degenerate preconditioned proximal point algorithms. We establish weak convergence results under mild assumptions that can be easily employed in the context of splitting methods for monotone inclusion and convex minimization problems. Moreover, we show that the degeneracy of the preconditioner allows for a reduction of the variables involved in the iteration updates. We show the strength of the proposed framework in the context of splitting algorithms, providing new simplified proofs of convergence and highlighting the link between existing schemes, such as Chambolle-Pock, Forward Douglas-Rachford and Peaceman-Rachford, that we study from a preconditioned proximal point perspective. The proposed framework allows to devise new flexible schemes and provides new ways to generalize existing splitting schemes to the case of the sum of many terms. As an example, we present a new sequential generalization of Forward Douglas-Rachford along with numerical experiments that demonstrate its interest in the context of nonsmooth convex optimization.
\end{abstract}



\section{Introduction}
\vspace{-0.05cm}
In this paper we study preconditioned proximal point methods where we allow the preconditioner to be degenerate, i.e.,~only positive semi-definite with a possibly large kernel. Our main motivation to do so is to apply the results to various splitting methods for optimization problems and monotone inclusions. We will show that these preconditioners still lead to convergent methods under mild assumptions and, moreover, that this allows to reduce the number of iteration variables in the context of splitting methods. The inclusion problem we consider is
\begin{equation}\label{monotoneinclusionIntro}
\text{find} \ u \in \bH \ \text{such that:} \ 0 \in \mathcal{A}u
\end{equation}
where $\bH$ is a real Hilbert space and $\mathcal{A}$ is a (in general multivalued) operator from $\bH$ into itself~\cite{BCombettes}. The set of solutions to \cref{monotoneinclusionIntro} is referred as \textit{zeros} of $\mathcal{A}$ and denoted by $\zer \mathcal{A}$. By the Minty surjectivity theorem \cite{MintyThm}, the resolvent $J_{\mathcal{A}}:=(I+\mathcal{A})^{-1}$ is a well-defined function from $\bH$ into itself as soon as $\mathcal{A}$ is maximal monotone; moreover it is firmly non-expansive and its fixed points correspond exactly to solutions of \cref{monotoneinclusionIntro} (cf.~\cite[Section 23]{BCombettes}). Hence, by the Krasnoselskii-Mann convergence theorem, the proximal point iteration $u^{k+1} = J_{\mathcal{A}}u^{k}$ weakly converges to a solution of \cref{monotoneinclusionIntro}. However, the evaluation of the resolvent is, in general, not much simpler that solving the inclusion problem itself. This changes, if one considers a \emph{preconditioner}, i.e.,~a linear, bounded, self-adjoint and positive semi-definite operator $\mathcal{M}:\bH\to\bH$. We have the equivalence of \cref{monotoneinclusionIntro} to
\begin{equation}\label{FixedPointInclusionIntro}
\text{find} \ u \in \bH \ \text{such that:} \ u \in \left(\mathcal{M}+\mathcal{A}\right)^{-1}\mathcal{M}u.
\end{equation}
Assuming that the operator $\mathcal{T}:=\left(\mathcal{M}+\mathcal{A}\right)^{-1}\mathcal{M}$ has full domain and is single-valued, one can turn the fixed-point inclusion into the preconditioned proximal point (\textsf{PPP}) iteration
\begin{equation}\label{eq:PPPsimple}
u^0 \in \bH, \quad u^{k+1}=\mathcal{T}u^k = \left( \mathcal{M}+\mathcal{A}\right)^{-1}\mathcal{M}u^k,
\end{equation}
and proper choices of $\mathcal{M}$ will allow for efficient evaluation of $\mathcal{T}$.

In general, $\mathcal{T}$ is an instance of the so-called \emph{warped} resolvents, where $\mathcal{M}$ is replaced by a general non-linear operator \cite{WarpedIterations}. Other particular instances of warped resolvents can be found under the name of $F$-resolvents \cite{bauschke2008general}, or $D$-resolvents \cite{Dresolv1, BregMonOptAlg}. In our framework, $\mathcal{M}$ will be always referred as \textit{preconditioner}.
In case of $\mathcal{M}=I$, we retrieve the usual proximal point iteration \cite{ProxPointClassicMartinet, Rockafellar1976MonotoneOA}, and if $\mathcal{M}$ defines an equivalent metric (i.e.,~$\mathcal{M}$ is strongly positive definite) then we simply get a proximal point iteration with respect to the operator $\mathcal{M}^{-1}\mathcal{A}$ in the Hilbert space endowed with the metric induced by $\mathcal{M}$. In this case, the convergence analysis of \cref{eq:PPPsimple} is essentially the same as for the case $\mathcal{M}=I$. 

In this work we consider the case where the preconditioner $\mathcal{M}$ is only positive semi-definite, and hence, $\mathcal{M}$ is associated with a semi inner-product $\langle u, v \rangle_{\mathcal{M}}:=\langle \mathcal{M}u, v \rangle$ (defined for all $u$ and $v$ in $\bH$), as well as a continuous seminorm $\| u\|_\mathcal{M} := \sqrt{\langle u, u \rangle_{\mathcal{M}}}$. Such preconditioner is typically involved in the context of splitting methods, that is, for inclusions of the form 
\begin{equation}\label{sum}
\textit{find} \ x \in \mathcal{H} \ \textit{such that}: \    0 \in (A_0+\dots+A_N)x,
\end{equation}
that we aim to solve using only the resolvents of each maximal monotone operator $A_i$ combined with simple algebraic operations. If $N=1$, one of the most popular splitting algorithm to tackle \cref{sum} is the celebrated Douglas-Rachford splitting (\textsf{DRS}) method \cite{douglas1956on, EcksteinBertsekas}, which, as already remarked in \cite{BrediesDRS}, admits a genuinely degenerate \textsf{PPP} representation with respect to the operators
\begin{equation}\label{DouglasRachfordSettingIntro}
\quad \mathcal{A}:= 
\begin{bmatrix}
\sigma A & I \\
-I & \left(\sigma B\right)^{-1}
\end{bmatrix},
\quad
\mathcal{M}:= 
\begin{bmatrix}
I & -I \\
-I & I\\
\end{bmatrix},
\end{equation}
with $\sigma > 0$, $A=A_0, \ B=A_1$. Let $u:=(x,y)\in \bH:= \mathcal{H}^{2}$, the inclusion problem $0 \in \mathcal{A}u$ in the larger space $\bH$ is indeed equivalent to $0 \in (A+B)x$, in fact $0 \in (A+B)x$ if and only if there exists $y \in \mathcal{H}$ such that $0 \in \sigma Ax+y$ and $y \in \sigma Bx$. The preconditioner has been chosen to be positive semi-definite and such that the \textsf{PPP} iterations can be computed explicitly. Indeed, $\mathcal{A}+\mathcal{M}$ has a lower triangular structure and we easily get
\begin{equation}\label{intro:DRSasPPP}
  \begin{cases}
    x^{k+1} = J_{\sigma A}(x^k-y^k), \\
    y^{k+1} = J_{(\sigma B)^{-1}}\bigl(2x^{k+1}-(x^k-y^k)\bigr).
  \end{cases}
\end{equation}
We can already notice that we only need the information contained in $x^k-y^k$, which leads to the substitution $w^k := x^k-y^k$ resulting exactly in the \textsf{DRS} iteration
\begin{equation}\label{intro:DRSasRPPP}
w^{k+1}=w^k+J_{\sigma B}(2  J_{\sigma A}w^k-w^k)- J_{\sigma A}w^k.
\end{equation}

The sequence $\{w^k\}_k$ can be shown to converge weakly to a point $w^*$ such that $J_{\sigma A}(w^*)$ is a solution of $0 \in (A+B)x$, provided such a point exists \cite{EcksteinBertsekas}. Notice, moreover, that passing from \cref{intro:DRSasPPP} to \cref{intro:DRSasRPPP} we reduced the variables from two to one. We show in this paper that this is not accidental. 

The convergence of the sequence $u^k=(x^k,y^k)$ is more subtle because in this case we are really dealing with a strongly degenerate preconditioner (as the preconditioner $\mathcal{M}$ has a large kernel), and convergence can not be deduced from the convergence of the proximal point method~\cite{BrediesDRS}. Instead, an \textit{ad hoc} proof of convergence that exploits the particular structure of the problem is given for instance in \cite{SVAITER2019291}.

\subsection{Motivation}
\label{sec:motivation}
Preconditioned proximal point iterations are attracting an increasing interest due to their fruitful applications in splitting algorithms for convex minimization and monotone inclusion problems, but the present lack of a general and sufficiently simple procedure to tackle the convergence and stability analysis in the degenerate case typically leads to involved and long proofs. Indeed, no definitive treatments can be found in the literature.

This motivates the development of a unifying theory that can give a new perspective on existing splitting schemes, simplifying the analysis of convergence, and paving the way for new algorithms that can also allow tackling objective functions composed by many terms.

\subsection{Related work}\label{sec:related-work}
The proximal point algorithm has an old history that dates back to the seminal works of Martinet and Rockafellar \cite{Rockafellar1976MonotoneOA, ProxPointClassicMartinet}. Soon after, the work of Cohen \cite{Cohen1980} introduces a closely related algorithmic approach referred as \textit{auxiliary problem principle}, which shows the seminal idea to replace the usual norm with generalized, even non-linear, variants. More recently, in the context of \textsf{PPP} as defined in \cref{eq:PPPsimple} we mention the work of He and Yuan \cite{HeYuan}, where the authors firstly propose the analysis of the Primal-Dual Hybrid Gradient method, also known as Chambolle-Pock (\textsf{CP}) algorithm \cite{Chambolle2011}, from a \textsf{PPP} standpoint, but a non-degenerate preconditioner is considered. A more general instance can be found in \cite{giselsson2021nonlinear} for solving the three-operator problem (problem \cref{sum} with $N=2$), but the analysis relies on non-degenerate preconditioners. Other instances of non-degenerate, but even non-linear, preconditioners lead to the notion of $D$-resolvent \cite{Dresolv1,bauschke2008general}, where the square distance in the usual definition of proximal point is replaced by a Bregman divergence generated by a strictly convex function resulting in a \textsf{PPP}-like iteration where $\mathcal{M}$ is replaced by the gradient of the generating function. More general resolvents can be found under the name of warped resolvents \cite{WarpedIterations}, where the preconditioner is replaced by a non-linear operator. %

For the degenerate case we mention the general treatment proposed by Valkonen in \cite{Valkonen2020} where the analysis of convergence includes \textsf{PPP} iterations, but do not go beyond the weak convergence of \textsf{PPP} in the $\mathcal{M}$ seminorm. Similar results can be found in the recent work by Bredies and Sun on the preconditioned variants of \textsf{DRS} and the alternating directions method of multipliers (\textsf{ADMM}), where authors also recover the convergence in the original norm, but with \textit{ad hoc} proofs of convergence \cite{BrediesDRS, Bredies2017APP}. Also, the so-called asymmetric forward–backward–adjoint splitting proposed in \cite{LafatPuyaPP} can be thought as a \textsf{PPP} instance with an additional forward term, but the degenerate-case analysis (with no forward term) is performed, with an arguably involved proof, only in finite dimensions. In \cite[Section 5]{Condat2020ProximalSA}, the authors propose a generalized \textsf{CP} algorithm presenting the scheme directly as a \textsf{PPP} method. The convergence is established in Theorem 5.1 in the non-degenerate case relying on the non-degenerate \textsf{PPP} formulation, and in a separate result (Theorem 5.2) for the degenerate case, after noticing that it is an instance of \textsf{DRS} on a modified space, which allows relying on existing proofs of convergence for \textsf{DRS}. Currently, particular cases of \textsf{PPP} are dealt with on a case-by-case basis. To the best of the authors' knowledge, a unified treatment does not exist in the literature.

\subsection{Contribution}\label{sec:contributio}

The paper is organized as follows. In \cref{sec:abstractPPP} we present the analysis of the degenerate case of the preconditioned proximal point method providing a systematic scheme to establish convergence (with and without rates) that only involves mild assumptions. Moreover, we also show how the degeneracy of the preconditioner allows for a reduction of the number of variables involved in the iteration updates. In \cref{sec:applications} we focus on the application to splitting algorithms. In this context, the main contribution that the proposed analysis provides is two-fold. First, it allows to simplify, in a unifying fashion, the convergence analysis for a variety of splitting methods, such as Chambolle-Pock in the limiting case, forward \textsf{DRS} (also called Davis-Yin method), and Peaceman-Rachford, that we study from a \textsf{PPP} perspective. Second, it allows to derive new splitting algorithms for the $(N+1)$-operator problem \cref{sum} that do not rely on the so-called \textit{product space trick} (cf. \cref{sec:ParallelFDR}) such as a \textit{sequential generalization of forward \textsf{DRS}}. In \cref{sec:Experiment} we present numerical experiments.

\section{Abstract degenerate preconditioned proximal point}
\label{sec:abstractPPP}
Let $\bH$ be a real Hilbert space, $\mathcal{A}:\bH\to 2^{\bH}$ be a (multivalued) operator (that through the rest of this paper we often identify with its graph in $\bH\times \bH$) and let $\mathcal{M}:\bH\to\bH$ be a linear bounded operator. Recall from \cref{FixedPointInclusionIntro} that \cref{monotoneinclusionIntro} could be formulated as a fixed point equation $u \in \mathcal{T}u$, with $\mathcal{T}:=\left(\mathcal{M}+\mathcal{A}\right)^{-1}\mathcal{M}$. Let us fix the class of preconditioners that we consider in this paper.
\begin{defi}[Admissible preconditioner]\label{def:admissiblePrec} An admissible preconditioner for the operator $\mathcal{A}:\bH\to 2^{\bH}$ is a linear, bounded, self-adjoint and positive semi-definite operator $\mathcal{M}:\bH\to \bH$ such that
	\begin{equation}\label{HpAdmissiblePrec}
	\mathcal{T} \ \text{is single-valued and has full domain}.
	\end{equation}
\end{defi}
\begin{remark}
Contrarily to the non-degenerate case, assuming for instance that $\mathcal{A}$ is maximal monotone does not necessarily lead to \cref{HpAdmissiblePrec}. As a non-trivial counterexample consider
\begin{equation*}
\bH=\mathbb{R}^2, \ \mathcal{A}=\partial f, \ \text{where} \ f(x,y)=\max\{e^y-x,0\} \ \text{and} \ \mathcal{M}(x,y)=(x,0),
\end{equation*}
It is easy to check that $\mathcal{T}$ is neither everywhere defined ($\mathcal{T}0=\emptyset)$ nor single-valued. For this reason, instead of imposing the maximal monotonicity of $\mathcal{A}$, we directly require \cref{HpAdmissiblePrec}, which, in the context of splitting methods, is a reasonable assumption. 
\end{remark}
Based on classical results of functional operator theory we can characterize an admissible preconditioner with the following fundamental decomposition. The proof is postponed to the appendix. 
\begin{prop}\label{CholeskyDecomposition}
	Let $\bH$ be a real Hilbert space, and $\mathcal{M}:\bH\to \bH$ be a linear, bounded, self-adjoint and positive semi-definite operator. Then there exists a bounded and injective operator $\mathcal{C}:\bD\to \bH$, for some real Hilbert space $\bD$, such that $\mathcal{M}=\mathcal{C}\mathcal{C}^*$. Moreover, if $\mathcal{M}$ has closed range, then $\mathcal{C}^*$ is onto.
\end{prop}
Now we state the algorithm we aim to analyze:
\paragraph{Degenerate preconditioned proximal point method} Let $\mathcal{A}:\bH \to 2^{\bH}$ be an operator, $\mathcal{M}:\bH\to \bH$ be an admissible preconditioner, and $\mathcal{T} = (\mathcal{M} + \mathcal{A})^{-1}\mathcal{M}$. Let $\{\lambda_k\}_k$ be a sequence in $[0,2]$ such that $\sum_{k \in \mathbb{N}} \lambda_k(2-\lambda_k)=+\infty$. Let 
\begin{equation}\label{eq:PPP}
u^0 \in \bH, \quad u^{k+1}=u^k+\lambda_k(\mathcal{T}u^k-u^k).
\end{equation}
In the degenerate case one may still consider $\mathcal{M}^{-1}$ as a multivalued operator and give a precise meaning to the composition $\mathcal{M}^{-1}\mathcal{A}$ as
\begin{equation*}
\mathcal{M}^{-1}\mathcal{A}u= \bigcup_{y \in \mathcal{A}u}\mathcal{M}^{-1}y = \big\{ z \in \bH \ | \ z \in \mathcal{M}^{-1}y, \ y \in \mathcal{A}u\big\}
\quad \text{for all} \ u \in \bH.
\end{equation*}
With this notation we have for all $(x,y)\in \bH^2$ that
\begin{equation*}
\begin{aligned}
\left(\mathcal{M}+\mathcal{A}\right)y \ni \mathcal{M}x &\iff \mathcal{M}(x-y)\in \mathcal{A}y \iff x-y \in \mathcal{M}^{-1}\mathcal{A}y\\
& \iff x \in \left(I+\mathcal{M}^{-1}\mathcal{A}\right)y \iff y \in \left(I+\mathcal{M}^{-1}\mathcal{A}\right)^{-1}x.
\end{aligned}
\end{equation*}
Therefore, the operator $\mathcal{T}$ coincides with the resolvent of $\mathcal{M}^{-1}\mathcal{A}$, i.e.,
\begin{equation}\label{Tdefinition}
\mathcal{T}=(\mathcal{M}+\mathcal{A})^{-1}\mathcal{M}=\left(I+\mathcal{M}^{-1}\mathcal{A}\right)^{-1},
\end{equation}
which is single-valued and with full domain by the admissibility of the preconditioner. Notice that if $\mathcal{M}$ is singular, say not surjective, the behaviour of $\mathcal{A}$ outside the image of $\mathcal{M}$ is somehow neglected by $\mathcal{M}^{-1}\mathcal{A}$. Indeed, we have the following list of equivalences
\begin{equation}\label{M-1AeA}
    \begin{aligned}
        (x,y)\in \mathcal{M}^{-1}\mathcal{A} \ \iff \ (x,\mathcal{M}y) \in \mathcal{A} \ &\iff \ (x, \mathcal{M}y)\in \mathcal{A}\cap (\bH\times \Image\mathcal{M})  \\
        &\iff \ (x,y) \in \mathcal{M}^{-1}(\mathcal{A}\cap (\bH\times \Image\mathcal{M})).
    \end{aligned}
\end{equation}
This suggests that standard demands on the structure of $\mathcal{A}$, e.g.,~the maximal monotonicity, can be relaxed without affecting the well-posedness of $\mathcal{T}$. The following notion will be crucial.
 \begin{defi}[$\mathcal{M}$-monotonicity] Let $\mathcal{M}:\bH\to\bH$ be a bounded linear positive semi-definite operator then $\mathcal{B}:\bH\to 2^{\bH}$ is $\mathcal{M}$-monotone if
        \begin{equation*}
        \langle v-v',u-u'\rangle_{\mathcal{M}}\geq 0 \quad \text{for all} \ (u,v),(u',v')\in \mathcal{B}.
        \end{equation*}    
    \end{defi}
As for the non-degenerate case, we will largely rely on the $\mathcal{M}$-monotonicity of $\mathcal{M}^{-1}\mathcal{A}$, which in terms of $\mathcal{A}$ can be explicitly characterized with the following.

  \begin{prop}\label{prop:MMaximalMonotonicityProp}
    Let $\mathcal{A}:\bH\to2^{\bH}$, and $\mathcal{M}:\bH\to\bH$ be a linear, bounded, self-adjoint and positive semi-definite operator. Then we have that $\mathcal{M}^{-1}\mathcal{A}$  is $\mathcal{M}$-monotone if and only if $\mathcal{A}\cap (\bH\times \Image\mathcal{M})$ is monotone.
  \end{prop}

The proof is postponed to the appendix. 

\subsection{On the convergence of \textsf{PPP}}
In this section we provide a proof of the weak convergence of \textsf{PPP} according to \cref{eq:PPP} under the additional assumption that $(\mathcal{M}+\mathcal{A})^{-1}$ is Lipschitz,
a condition often satisfied in the context of splitting algorithms (cf. \cref{sec:applications}).

It has already been observed (see, e.g., \cite{BrediesDRS, WarpedIterations}) that some important notions can be easily generalized to the degenerate case, such as the firm non-expansiveness.

\begin{lemma}\label{lem:Mfne}
	Let $\mathcal{A}:\bH\to 2^{\bH}$ be an operator and $\mathcal{M}$ an admissible preconditioner such that $\mathcal{M}^{-1}\mathcal{A}$ is $\mathcal{M}$-monotone. Then $\mathcal{T}$ is $\mathcal{M}$-firmly non-expansive, i.e.,
	\begin{equation}\label{MFirmlyNonExpansive}
	\| \mathcal{T}u_1-\mathcal{T}u_2\|_\mathcal{M}^2+\|(I-\mathcal{T})u_1-(I-\mathcal{T})u_2 \|^2_\mathcal{M}\leq \|u_1-u_2\|^2_\mathcal{M}\quad \text{for all} \ u_1,u_2\in \bH.
	\end{equation}
\end{lemma}
For a proof see \cite[Lemma 2.5]{BrediesDRS}. There, monotonicity of $\mathcal{A}$ is assumed, but exactly the same proof holds even if we only assume the monotonicity of $\mathcal{A}\cap (\bH\times \Image\mathcal{M})$. 

\begin{remark}\label{remark:R}
	As in the non-degenerate case, one can show that an operator $\mathcal{T}$ is $\mathcal{M}$-firmly non-expansive if and only if it is $\mathcal{M}$-$(1/2)$-averaged, that is, $\mathcal{R}:=2\mathcal{T}-I$ is $\mathcal{M}$-nonexpansive, i.e.,
    \[\|\mathcal{R}u_1-\mathcal{R}u_2\|_\mathcal{M}\leq \|u_1-u_2\|_\mathcal{M} \quad \text{for all} \ u_1,u_2\in\bH.\]
	This follows in an obvious way from the parallelogram identity w.r.t.~$\mathcal{M}$, since for any $u, u' \in \bH$,
	\begin{align*}
	\|&\mathcal{R}u_1-\mathcal{R}u_2\|_\mathcal{M}^2 = \| \mathcal{T}u_1-u_1-(\mathcal{T}u_2-u_2)+\mathcal{T}u_1-\mathcal{T}u_2\|^2_\mathcal{M}\\
	&=2\|(I-\mathcal{T})u_1-(I-\mathcal{T})u_2\|_\mathcal{M}^2+2\|\mathcal{T}u_1-\mathcal{T}u_2\|^2_\mathcal{M}-\|u_1-u_2\|_\mathcal{M}^2 \leq \|u_1-u_2\|^2_\mathcal{M}
	\end{align*}
	is equivalent to $\|(I-\mathcal{T})u_1-(I-\mathcal{T})u_2\|_\mathcal{M}^2+\|\mathcal{T}u_1-\mathcal{T}u_2\|_\mathcal{M}^2 \leq \|u_1-u_2\|_\mathcal{M}^2.$
\end{remark}

From the $\mathcal{M}$-firm non-expansiveness \cref{MFirmlyNonExpansive} we can derive the analogous in the $\mathcal{M}$-seminorm of the asymptotic regularity and the Fejér monotonicity of $\{u^k\}_k$ w.r.t.~the fixed points of $\mathcal{T}$, which we denote by $\Fix \mathcal{T}$.

\begin{lemma}\label{lem:MfejerMon}
	Let $\mathcal{A}:\bH\to 2^{\bH}$ be an operator such that $\zer\mathcal{A}\neq \emptyset$ and $\mathcal{M}$ an admissible preconditioner such that $\mathcal{M}^{-1}\mathcal{A}$ is $\mathcal{M}$-monotone. Let $\{u^k\}_k$ be the sequence generated by \textsf{PPP} according to \cref{eq:PPP}. Then, $\{u^k\}_k$ is $\mathcal{M}$-asymptotically regular and $\mathcal{M}$-Fejér monotone with respect to $\Fix\mathcal{T}$, i.e.,~we have
	\begin{equation*}
	\lim_{k\to +\infty}\| \mathcal{T}u^{k}-u^k\|_\mathcal{M}=0, \quad \| u^{k+1}-u\|_\mathcal{M}\leq \|u^k-u\|_\mathcal{M} \quad \text{for all} \ k \in \mathbb{N}, \ \textit{and} \ u \in \Fix\mathcal{T}.
	\end{equation*}
\end{lemma}
\begin{proof}
	First, notice that for all $\alpha \in \mathbb{R}$ and for all $u, \ u' \in \bH$ we have
	\begin{equation}\label{MStrictConvex}
	\|\alpha u+(1-\alpha)u'\|_\mathcal{M}^2+\alpha(1-\alpha)\|u-u'\|^2_\mathcal{M}=\alpha \|u\|_\mathcal{M}^2+(1-\alpha)\|u'\|_\mathcal{M}^2.
	\end{equation}
	By \cref{lem:Mfne}, the operator $\mathcal{T}$ is $\mathcal{M}$-firmly non-expansive, hence, $\mathcal{R}:=2\mathcal{T}-I$ is $\mathcal{M}$-non-expansive, see \cref{remark:R}. Observe that $\Fix\mathcal{T}=\Fix\mathcal{R}$ and that, putting $\mu_k=\lambda_k/2$, we have,
	\begin{equation*}
	u^{k+1}=u^k+\mu_k(\mathcal{R}u^k-u^k) \quad \text{for all} \ k \in \mathbb{N}.
	\end{equation*}
	It follows from \cref{MStrictConvex} and the $\mathcal{M}$-non-expansiveness of $\mathcal{R}$ that for every $u \in \Fix\mathcal{T}$ and every $k \in \mathbb{N}$
	\begin{equation}\label{MfejerMasyProof}
	\begin{aligned}
	\|u^{k+1}&-u\|^2_\mathcal{M}=\|(1-\mu_k)(u^k-u)+\mu_k(\mathcal{R}u^k-u)\|_\mathcal{M}^2\\
	&=(1-\mu_k)\|u^k-u\|_\mathcal{M}^2+\mu_k\|\mathcal{R}u^k-\mathcal{R}u\|_\mathcal{M}^2-\mu_k(1-\mu_k)\|\mathcal{R}u^k-u^k\|_\mathcal{M}^2\\
	&\leq \|u^k-u\|_\mathcal{M}^2-\mu_k(1-\mu_k)\|\mathcal{R}u^k-u^k\|_\mathcal{M}^2.
	\end{aligned}
	\end{equation}
	Hence, $\{u^k\}_k$ is $\mathcal{M}$-Fejér monotone with respect to $\Fix\mathcal{T}$. For the $\mathcal{M}$-asymptotic regularity, we derive from \cref{MfejerMasyProof} that
	\begin{equation}\label{eq:summability}
	\sum_{k \in \mathbb{N}}\mu_k(1-\mu_k)\|\mathcal{R}u^k-u^k\|_\mathcal{M}^2\leq \|u^0-u\|_\mathcal{M}^2.
	\end{equation}
	Since $\sum_{k \in \mathbb{N}}\mu_k(1-\mu_k)=\tfrac14\sum_{k \in \mathbb{N}}\lambda_k(2-\lambda_k)=+\infty$, we have $\lim \inf_k\|\mathcal{R}u^k-u^k\|_\mathcal{M}=0$. However, for every $k \in \mathbb{N}$,
	\begin{equation*}
	\begin{aligned}
	\|\mathcal{R}u^{k+1}-u^{k+1}\|_\mathcal{M}=\|\mathcal{R}u^{k+1}-\mathcal{R}u^{k}+(1-\mu_k)(\mathcal{R}u^k-u^k)\|_\mathcal{M}\\
	\leq \|u^{k+1}-u^k\|_\mathcal{M}+(1-\mu_k)\|\mathcal{R}u^k-u^k)\|_\mathcal{M}=\|\mathcal{R}u^k-u^k\|_{\mathcal{M}}.
	\end{aligned}
	\end{equation*}
	Consequently, $\{\|\mathcal{R}u^k-u^k\|_\mathcal{M}\}_{k}$ converges and we have $\lim_{k\to +\infty}\|\mathcal{R}u^k-u^k\|_\mathcal{M} = 0$, that implies in particular $\lim_{k\to +\infty}\|\mathcal{T}u^{k}-u^k\|_\mathcal{M} = 0$.
\end{proof}
\begin{theorem}[Convergence]
	\label{WeakConvergencePPPnew}
	Let $\mathcal{A}:\bH\to 2^{\bH}$ be an operator with $\zer\mathcal{A}\neq \emptyset$ and $\mathcal{M}$ an admissible preconditioner such that $\mathcal{M}^{-1}\mathcal{A}$ is $\mathcal{M}$-monotone and $(\mathcal{M}+\mathcal{A})^{-1}$ is $L$-Lipschitz. Let $\{u^k\}_k$ be the sequence generated by \textsf{PPP} according to \cref{eq:PPP}.
	If every weak cluster point of $\{\mathcal{T}u^k\}_k$ lies in $\Fix\mathcal{T}$, then the sequence $\{\mathcal{T}u^k\}_k$ converges weakly in $\bH$ to some $u^* \in \Fix\mathcal{T}$. Furthermore, if $0<\inf_k \lambda_k \leq \sup_k \lambda_k < 2$ also the sequence $\{u^k\}_k$ converges weakly to $u^*$.
\end{theorem}
\begin{proof}
	Let $\mathcal{M}=\mathcal{C}\mathcal{C}^*$ be a decomposition of $\mathcal{M}$ according to \cref{CholeskyDecomposition}. First, since $(\mathcal{M}+\mathcal{A})^{-1}$ is $L$-Lipschitz and $\|\mathcal{C}^*u\|=\|u\|_\mathcal{M}$ for every $u \in \bH$, we have for all $u', \ u'' \in \bH$ 
	\begin{equation}\label{eq:step1}
	\|\mathcal{T}u'-\mathcal{T}u''\|=\|(\mathcal{M}+\mathcal{A})^{-1}\mathcal{C}\mathcal{C}^*u'-(\mathcal{M}+\mathcal{A})^{-1}\mathcal{C}\mathcal{C}^*u''\|\leq L\|\mathcal{C}\| \|u'-u''\|_{\mathcal{M}}.
	\end{equation}
	This property, combined with the $\mathcal{M}$-Fejér monotonicity of $\{u^k\}_k$ given by \cref{lem:MfejerMon} yields the boundedness of the sequence $\{\mathcal{T}u^k\}$. Indeed, for all $u \in \Fix\mathcal{T}$ we have
	\begin{equation*}
	\|\mathcal{T}u^k-u\|\leq C\|u^k-u \|_\mathcal{M} \leq C \|u^0-u \|_\mathcal{M},
	\end{equation*}
	and thus, $\{\mathcal{T}u^k\}_k$ is bounded. Furthermore, for all $u\in \Fix \mathcal{T}$, we have
	\begin{equation}\label{eq:convergenceofTuk}
	    \|\mathcal{T}u^k-u\|^2_\mathcal{M}=\|\mathcal{T}u^k-u^k\|^2_\mathcal{M}+\|u^k-u\|^2_\mathcal{M}+2\langle \mathcal{T}u^k-u^k,u^k-u \rangle_{\mathcal{M}}.
	    \end{equation}
    The term $\langle\mathcal{T}u^k-u^k,u^k-u \rangle_{\mathcal{M}}$ in \cref{eq:convergenceofTuk} converges to $0$ by the Cauchy-Schwarz inequality w.r.t.~the $\mathcal{M}$-scalar product and \cref{lem:MfejerMon}. Thus, we obtain, thanks to \cref{lem:MfejerMon}, that for all $u \in \Fix \mathcal{T}$ the sequence $\{\|\mathcal{T}u^k-u\|_\mathcal{M}\}_k$ converges.
	
	To conclude the proof we use Opial's argument adapted to our degenerate context. For any fixed point $u^*$ we define $\ell(u^*)$ to be the limit of $\|\mathcal{T}u^k-u^*\|_\mathcal{M}$. Since the sequence $\{\mathcal{T}u^k\}_k$ is bounded, has at least one weak cluster point. Let $\{\mathcal{T}u^{k_i}\}_i$ be a sequence converging to this cluster point $u^*$. By assumption, every weak cluster point of $\{\mathcal{T}u^k\}_k$ lies in $\Fix\mathcal{T}$, hence, $u^*=\mathcal{T} u^*$. Let $u^{**}$ be another weak cluster point of $\{\mathcal{T}u^k\}_k$ (which implies $u^{**}=\mathcal{T} u^{**}$) and $\{\mathcal{T}u^{l_i}\}_i$ be the sequence weakly converging to $u^{**}$. If we suppose $\|u^*- u^{**}\|_\mathcal{M}>0$, then Opial's property would give us
	\begin{equation*}
	\begin{aligned}
	\liminf_{i\to\infty}\|\mathcal{T}u^{k_i}-u^*\|_\mathcal{M}&<\liminf_{i\to\infty}\|\mathcal{T}u^{k_i}-u^{**}\|_\mathcal{M},\\
	\liminf_{i\to\infty}\|\mathcal{T}u^{l_i}-u^{**}\|_\mathcal{M}&<\liminf_{i\to\infty}\|\mathcal{T}u^{l_i}-u^{*}\|_\mathcal{M},
	\end{aligned}
	\end{equation*}
	so that we get both $\ell(u^{*})<\ell(u^{**})$ and $\ell(u^{**})<\ell(u^{*})$, which is a contradiction and thus, $\|u^{*}-u^{**}\|_{\mathcal{M}}=0$. Therefore, $\mathcal{M}u^*=\mathcal{M}u^{**}$ and in particular $u^*=\mathcal{T}u^*=\mathcal{T}u^{**}=u^{**}$ by construction of $\mathcal{T}$. Thus, the iterates $\{\mathcal{T}u^k\}_k$ weakly converge to a fixed point of $\mathcal{T}$.
	
	Let us turn our attention to the sequence $\{u^k\}_k$. Under the additional assumption $0<\inf_k \lambda_k \leq \sup_k \lambda_k < 2$, say $\lambda_k \in [\epsilon, 2-\epsilon]$ for some $\epsilon > 0$, we can prove that $\{\mathcal{T}u^{k}-u^{k}\}_k$ converges to zero strongly, so that, since $\{\mathcal{T}u^{k}\}_k$ converges weakly to a fixed point of  $\mathcal{T}$, also $\{u^{k}\}_k$ would converge weakly to the same point. Let us call $\phi_k:=\|\mathcal{T}u^{k}-u^{k}\|^2$. We have
	\begin{equation*}
	\begin{aligned}
	\phi_{k+1}=\| &\mathcal{T}u^{k+1}-\mathcal{T}u^k-(1-\lambda_k)(u^k-\mathcal{T}u^k)\|^2 = \| \mathcal{T}u^{k+1}-\mathcal{T}u^k\|^2\\
    &+|1-\lambda_k|^2\|\mathcal{T}u^k-u^k\|^2+2(1-\lambda_k)\langle \mathcal{T}u^{k+1}-\mathcal{T}u^k, \mathcal{T}u^k-u^k\rangle.
    \end{aligned}
	\end{equation*}
	Therefore, using \cref{eq:step1}, the fact that $|1-\lambda_k|\leq 1-\epsilon$ for all $k$, and choosing $\delta> 0$ such that $\eta:=(1-\epsilon)^2+\delta < 1$ we get
	\begin{equation*}
	\begin{aligned}
	\phi_{k+1} &\leq \|\mathcal{T}u^{k+1}-\mathcal{T}u^k\|^2+(1-\epsilon)^2\phi_k+\left( \frac{1}{\delta}\|\mathcal{T}u^{k+1}-\mathcal{T}u^k\|^2+ \delta \phi_k\right)\\
	& \leq \eta \phi_k + C \|u^{k+1}-u^k\|_\mathcal{M}^2
	\end{aligned}
	\end{equation*}
	The relation $\phi_{k+1}\leq \eta \phi_k + \alpha_k$, with $\alpha_k = C\|u^{k+1}-u^k\|^2_\mathcal{M}$, which is summable when $\lambda_k \in [\epsilon, 2-\epsilon]$, see~\cref{eq:summability}, allows to conclude that also $\{\phi_k\}_k$ is summable. Indeed, induction yields $\phi_{k} \leq \eta^{k}\phi_0+\sum_{j=0}^k \eta^{j}\alpha_{k-j}$, whose right-hand side is summable being the sum of a geometric sequence and the Cauchy product of two absolutely convergent series. Thus, $\{\phi_k\}_k$ is summable too and in particular $\phi_k \to 0$.
\end{proof}

\begin{corollary}\label{PPPwithmaxmon}
        Let $\mathcal{A}:\bH\to 2^{\bH}$ be a maximal monotone operator with $\zer\mathcal{A}\neq \emptyset$ and $\mathcal{M}$ an admissible preconditioner such that $(\mathcal{M}+\mathcal{A})^{-1}$ is $L$-Lipschitz. Let $\{u^k\}_k$ be the sequence generated by \textsf{PPP} according to \cref{eq:PPP}. Then $\{\mathcal{T}u^k\}_k$ converges weakly to some $u^* \in \Fix\mathcal{T}$, i.e.,~a zero of $\mathcal{A}$. Furthermore, if $0<\inf_k \lambda_k \leq \sup_k \lambda_k < 2$ also the sequence $\{u^k\}_k$ converges weakly to $u^*$.
    \end{corollary}
    \begin{proof}
    Since $\mathcal{A}$ is monotone, $\mathcal{M}^{-1}\mathcal{A}$ is $\mathcal{M}$-monotone.
        We just have to prove that every weak cluster point of $\{\mathcal{T}u^k\}_k$ belongs to $\Fix\mathcal{T}$. Assume $\mathcal{T}u^{k_i}\rightharpoonup u\in \bH$. Using the $\mathcal{M}$-asymptotic regularity of $\{u^k\}_k$ we have
    \begin{equation}\label{ConvergenzaAu}
        \mathcal{A}\mathcal{T}u^{k_i}\ni \mathcal{M}(u^{k_i}-\mathcal{T}u^{k_i})\longrightarrow 0.
    \end{equation}
    By the maximality of $\mathcal{A}$ we have that $\mathcal{A}$ is closed in $\bH_{\textit{weak}}\times \bH_{\textit{strong}}$ (see \cite[Proposition 20.38]{BCombettes}), hence $0 \in \mathcal{A}u$. In other words, $u$ is a fixed point of $\mathcal{T}$.
    \end{proof}

\begin{remark}\label{rem:ContinuityOfAinzero}
    We point out that in \cref{PPPwithmaxmon}, we assumed two crucial properties, namely the Lipschitz regularity of $(\mathcal{M}+\mathcal{A})^{-1}$ and the maximality of $\mathcal{A}$. The former, which is a mild assumption especially in applications to splitting algorithms, is employed, together with the boundedness of $\{u^k\}_k$ in the $\mathcal{M}$-seminorm, to establish the boundedness of $\{\mathcal{T}u^k\}_k$ in $\bH$. The latter is only assumed for its geometric consequences, namely that if for all $k \in \mathbb{N}$,  $(\xi^k,a^k)\in \mathcal{A}$ and $a^k \to 0$, $\xi^k\rightharpoonup \xi$, then $0 \in \mathcal{A}\xi$,
    since this guarantees that every weak cluster point of $\{\mathcal{T}u^k\}_k$ lies in $\Fix \mathcal{T}$. Other different assumptions could be considered, for instance in \cite{BrediesDRS, Bredies2017APP}, the authors rely on the demiclosedness principle, which, however, could be challenging to control and may not hold in general.
     \end{remark}

\subsection{The reduced preconditioned proximal point algorithm}
The aim of splitting methods is to turn an inclusion $0\in Ax$ with possibly difficult resolvent for $A$ into an inclusion $0\in A_{0}x + \cdots + A_{N}x$ where the individual resolvents for the $A_{i}$ are cheaper to evaluate. This approach, however, increases the number of variables that one needs to store. As we have seen in the introduction, one can reduce the number of variables in the case of the \textsf{DRS} method, and perform an iteration in just a single variable. While this is not of great practical importance for \textsf{DRS}, it will be more important in the case of methods where the number $N$ is larger. In this section we will derive a systematic way to deduce a reduced version of degenerate preconditioned proximal point methods.

The main idea is that, for $u \in \bH$, to evaluate $\mathcal{T}u = (\mathcal{M}+\mathcal{A})^{-1}\mathcal{M}u$ we are actually interested in the information contained in $\mathcal{M}u$ which belongs to $\Image\mathcal{M}$, and for degenerate preconditioners this is typically a proper subspace of $\bH$. In this section we show how to exploit this fact to significantly reduce the number of variables involved in the iterations, resulting in a storage efficient, but still equivalent, method that we will call \textit{reduced preconditioned proximal point} (\textsf{rPPP}) algorithm. In the rest of this section we assume that $\mathcal{M}$ has closed range.

\begin{lemma}\label{MonotoneOperatorsPushForward}
	Let $\mathcal{A}:\bH\to 2^{\bH}$ be an operator, $\mathcal{M}:\bH\to\bH$ be an admissible preconditioner with closed range and $\mathcal{M}=\mathcal{C}\mathcal{C}^*$ be a decomposition of $\mathcal{M}$ according to \cref{CholeskyDecomposition}. The parallel composition $\mathcal{C}^* \rhd \mathcal{A}:=(\mathcal{C}^*\mathcal{A}^{-1}\mathcal{C})^{-1}$ can be expressed as
	\begin{equation*}
	\mathcal{C}^* \rhd \mathcal{A} = \{(\mathcal{C}^* x, \mathcal{C}^* y)\in \bH^2 \text{ s.t. } (x,y) \in \mathcal{M}^{-1}\mathcal{A}\}.
	\end{equation*}
\end{lemma}
\begin{proof}
	Let $\mathcal{B}$ be the right-hand side. If $(\bar{x},\bar{y}) \in \mathcal{B}$, then $\bar x = \mathcal{C^*}x$ and $\bar y = \mathcal{C^*}y$ for some $x,y\in\bH$ with $(x,y) \in \mathcal{M}^{-1}\mathcal{A}$. Then, $x \in \mathcal{A}^{-1}\mathcal{C}\mathcal{C}^*y$, hence, $\mathcal{C}^*x \in  \mathcal{C}^*\mathcal{A}^{-1}\mathcal{C}\mathcal{C}^*y$, thus $(\bar x, \bar y)=(\mathcal{C}^*x,\mathcal{C}^*y) \in \mathcal{C}^*\rhd \mathcal{A}$. Conversely, if $(\bar x,\bar y)\in \mathcal{C}^*\rhd \mathcal{A}$ then $\bar x \in \mathcal{C}^*\mathcal{A}^{-1}\mathcal{C}\bar y$. Since $\mathcal{C^*}$ is onto, there is $x,y\in\bH$ such that $\bar x=\mathcal{C^*}x,\ \bar y=\mathcal{C^*}y$ and hence $\mathcal{C^*}x\in \mathcal{C}^*\mathcal{A}^{-1}\mathcal{M}y$. This precisely means that there exists $\widetilde{x}\in \mathcal{A}^{-1}\mathcal{M}y$ such that $\mathcal{C}^{*}x = \mathcal{C}^*\widetilde{x}$, since $\widetilde{x} \in \mathcal{A}^{-1}\mathcal{M}y$ we have $(\widetilde{x},y) \in \mathcal{M}^{-1}\mathcal{A}$, thus $(\bar{x},\bar{y})=(\mathcal{C}^*\widetilde{x},\mathcal{C}^*y) \in \mathcal{B}$.
\end{proof}

We denote by $\widetilde{\mathcal{T}}$ the resolvent of $\mathcal{C}^*\rhd \mathcal{A}$, i.e., 
\begin{equation}\label{TtildeDefinition}
\widetilde{\mathcal{T}}:=\left(I+\mathcal{C}^*\rhd \mathcal{A}\right)^{-1}.
\end{equation}    We have the following fundamental theorem.
\begin{theorem}\label{TtildeTheorem}
	Let $\mathcal{A}$ be an operator on $\bH$ and $\mathcal{M}$ an admissible preconditioner with closed range such that $\mathcal{M}^{-1}\mathcal{A}$ is $\mathcal{M}$-monotone. Let $\mathcal{M}=\mathcal{C}\mathcal{C}^*$ be a decomposition of $\mathcal{M}$ according to \cref{CholeskyDecomposition} with $\mathcal{C}:\bD\to\bH$. Then, $\mathcal{C}^*\rhd \mathcal{A}$ is a maximal monotone operator and for $\widetilde{\mathcal{T}}$ according to \cref{TtildeDefinition} the following identity holds
	\begin{equation}\label{TtildeDefinitionEquivalence}
	\widetilde{\mathcal{T}}= \mathcal{C}^{*}(\mathcal{M}+\mathcal{A})^{-1}\mathcal{C}.
	\end{equation}
	In particular, $\widetilde{\mathcal{T}}:\bD\to\bD$ is everywhere well defined and firmly non-expansive.
\end{theorem}
\begin{proof}
	First, since $\mathcal{M}^{-1}\mathcal{A}$ is $\mathcal{M}$-monotone, $\mathcal{C}^*\mathcal{A}^{-1}\mathcal{C}$ is a monotone operator on $\bD$. Indeed, by \cref{MonotoneOperatorsPushForward}, for all $(x,y), \ (x',y')\in \mathcal{C}^*\rhd \mathcal{A}$ there exist $(u,v), \ (u',v')\in \mathcal{M}^{-1}\mathcal{A}$ such that $\mathcal{C}^*u=x, \ \mathcal{C}^*u'=x';\ \mathcal{C}^*v=y, \ \mathcal{C}^*v'=y'$. Therefore, using that $\mathcal{M}^{-1}\mathcal{A}$ is $\mathcal{M}$-monotone, we have
	\begin{equation*}
	\langle y-y',x-x'\rangle =    \langle \mathcal{C}^*v-\mathcal{C}^*v',\mathcal{C}^*u-\mathcal{C}^*u'\rangle = \langle u'-v',u-v\rangle_\mathcal{M}\geq 0.
	\end{equation*}
	This shows the monotonicity of $\mathcal{C}^*\rhd \mathcal{A}$, which is equivalent to the monotonicity of its inverse $\mathcal{C}^*\mathcal{A}^{-1}\mathcal{C}$. The rest of the proof is an application of a Woodbury-like identity:
	\begin{equation}\label{WoodburyMoreau}
	\left( I+\mathcal{C}^*\mathcal{A}^{-1}\mathcal{C}\right)^{-1}=I-\mathcal{C}^*\left(\mathcal{C}\mathcal{C}^*+\mathcal{A}\right)^{-1}\mathcal{C}.
	\end{equation}
	To prove this, we proceed as in \cite[Proposition 23.25]{BCombettes}. Let $u \in \bD$ and note that, by the surjectivity of $\mathcal{C}^*$ and the fact that $\mathcal{T}$ has full domain, $v =\left(\mathcal{C}\mathcal{C}^*+\mathcal{A}\right)^{-1}\mathcal{C}u$ is well defined. We prove that for $p:=u-\mathcal{C}^*v$ it holds $p=(I+\mathcal{C}^*\mathcal{A}^{-1}\mathcal{C})^{-1}u$. Since $\mathcal{C}u\in \mathcal{C}\mathcal{C}^*v+\mathcal{A}v$, we have
	\begin{equation*}
	\mathcal{C}p=\mathcal{C}\left(u-\mathcal{C}^*v\right)=\mathcal{C}u-\mathcal{C}\mathcal{C}^*v\in \mathcal{A}v.
	\end{equation*}
	Therefore, $v \in \mathcal{A}^{-1}\mathcal{C}p$, hence $u-p=\mathcal{C}^*v \in \mathcal{C}^*\mathcal{A}^{-1}\mathcal{C}p$. This proves that the right hand side in \cref{WoodburyMoreau} is contained in $\left( I+\mathcal{C}^*\mathcal{A}^{-1}\mathcal{C}\right)^{-1}$. It follows that the latter operator has full domain, and since $\mathcal{C}^*\mathcal{A}^{-1}\mathcal{C}$ is monotone, by Minty's theorem, $\mathcal{C}^*\mathcal{A}^{-1}\mathcal{C}$ is also maximal and, hence, its resolvent is single valued, and the two operators in \cref{WoodburyMoreau} coincide. Using the Moreau identity on $\left( I+\mathcal{C}^*\mathcal{A}^{-1}\mathcal{C}\right)^{-1}$ we get the desired result.
\end{proof}
We are now ready to state and prove the main theorem of this section, which in the degenerate case provides an equivalent reduced scheme for \textsf{PPP}.
\begin{theorem}[Reduction and convergence]\label{ReductionAndConvergence}
	Let $\mathcal{A}:\bH\to 2^{\bH}$ with $\zer\mathcal{A}\neq~\emptyset$ and let $\mathcal{M}$ be an admissible preconditioner with closed range such that $\mathcal{M}^{-1}\mathcal{A}$ is $\mathcal{M}$-monotone. Let  $\mathcal{M}=\mathcal{C}\mathcal{C}^*$ be a decomposition of $\mathcal{M}$ according to \cref{CholeskyDecomposition} with $\mathcal{C}:\bD\to\bH$ and denote by $\{u^k\}_k$ the sequence generated by \textsf{PPP} from \cref{eq:PPP}. Then, the following proximal point algorithm
	\begin{equation}\label{ReducedAlg}
		w^0=\mathcal{C}^*u^0, \quad w^{k+1}=w^k+\lambda_k(\widetilde{\mathcal{T}}w^k-w^k),
	\end{equation}
	is equivalent to \cref{eq:PPP}, in the sense that $w^k=\mathcal{C}^*u^k$ for all $k \in \mathbb{N}$. Moreover, $\{w^k\}_k$ weakly converges in $\bD$ to a point $w^*$ such that $\left(\mathcal{M}+\mathcal{A}\right)^{-1}\mathcal{C}w^*$ is a fixed point of $\mathcal{T}$.
\end{theorem}
\begin{proof}
	First, we prove that $w^k=\mathcal{C}^*u^k$ for all $k\in \mathbb{N}$. The case $k=0$ holds by assumption and using \cref{TtildeDefinitionEquivalence}, we get inductively that
	\begin{align*}
	w^{k+1}&=w^k+\lambda_k \bigl(\mathcal{C}^*(\mathcal{M}+\mathcal{A})^{-1}\mathcal{C}w^k-w^k \bigr)\\
	&=\mathcal{C}^*u^k+\lambda_k \bigl(\mathcal{C}^*(\mathcal{M}+\mathcal{A})^{-1}\mathcal{C}\mathcal{C}^*u^k-\mathcal{C}^*u^k \bigr)\\
	&=\mathcal{C}^{*} \bigl(u^k+\lambda_k \bigl( (\mathcal{M}+\mathcal{A})^{-1}\mathcal{M}u^k-u^k \bigr) \bigr)
	=\mathcal{C}^*u^{k+1}.
	\end{align*}
	To establish convergence of the reduced algorithm, since $\mathcal{C}^*\rhd \mathcal{A}$ is maximal monotone by \cref{TtildeTheorem}, we only need to prove that $\widetilde{\mathcal{T}}$ has fixed points (see for example \cite[Corollary 5.17]{BCombettes}). We actually have that $\mathcal{C}^*\Fix\mathcal{T}=\Fix \widetilde{\mathcal{T}}$. Indeed, given $u \in \Fix\mathcal{T}$ then $w=\mathcal{C}^*u$ is a fixed point of $\widetilde{\mathcal{T}}$, in fact, 
	\begin{equation*}
	\widetilde{\mathcal{T}}w=\mathcal{C}^*\left(\mathcal{M}+\mathcal{A}\right)^{-1}\mathcal{C}w=\mathcal{C}^*\left(\mathcal{M}+\mathcal{A}\right)^{-1}\mathcal{C}\mathcal{C}^*u=\mathcal{C}^*\mathcal{T}u=\mathcal{C}^*u=w.
	\end{equation*}
	Conversely, if $w$ is a fixed point of $\widetilde{\mathcal{T}}$ then to retrieve a fixed point of $\mathcal{T}$ we take $u:=\left( \mathcal{M}+\mathcal{A}  \right)^{-1}\mathcal{C}w$. Clearly $w=\mathcal{C}^*u$ and $u$ is a fixed point of $\mathcal{T}$ since
	\begin{equation*}
	\mathcal{T}u=\left( \mathcal{M}+\mathcal{A}  \right)^{-1}\mathcal{C}\mathcal{C}^*\left( \mathcal{M}+\mathcal{A}  \right)^{-1}\mathcal{C}w=\left( \mathcal{M}+\mathcal{A}  \right)^{-1}\mathcal{C}w=u.
	\end{equation*}
	Since $\Fix\mathcal{T}=\zer \mathcal{A}$, which is assumed non-empty, we have the desired result. 
\end{proof}
\begin{corollary}\label{remark:samesolution}
Let $\mathcal{A}:\bH\to2^{\bH}$ with $\zer \mathcal{A}\neq \emptyset$ be a maximal monotone operator and $\mathcal{M}:\bH\to\bH$ be an admissible preconditioner with closed range such that $(\mathcal{M}+\mathcal{A})^{-1}$ is Lipschitz. Let $\mathcal{M}=\mathcal{C}\mathcal{C}^*$ be a decomposition of $\mathcal{M}$ according to \cref{CholeskyDecomposition} with $\mathcal{C}:\bD\to\bH$ and denote by $\{w^k\}_k$ the sequence generated by \textsf{rPPP} from \cref{ReducedAlg}. Then $\{w^k\}_k$ converges weakly to a point $w^* \in \Fix \widetilde{\mathcal{T}}$ such that $u^*:=(\mathcal{M}+\mathcal{A})^{-1}\mathcal{C}w^* \in \Fix \mathcal{T}$. Moreover, $\{(\mathcal{M}+\mathcal{A})^{-1}\mathcal{C}w^k\}_k$ weakly converges to $u^*$.
\end{corollary}
\begin{proof}
	Thanks to \cref{ReductionAndConvergence} and \cref{PPPwithmaxmon}, we only need to prove that $\{\mathcal{T}u^k\}_k$ converges weakly to $u^{*}$ (notice that $\mathcal{T}u^k =(\mathcal{M}+\mathcal{A})^{-1}\mathcal{C}w^k$ for all $k$), where $\{u^k\}_k$ is the correspondent \textsf{PPP} sequence. Let $u^{**}\in \Fix \mathcal{T}$ be the weak limit of $\{\mathcal{T}u^k\}_k$ according to \cref{PPPwithmaxmon}. Since $\mathcal{C}^*$ is a linear bounded map, $\mathcal{C}^*u^{**}$ is the weak limit of $\mathcal{C}^*\mathcal{T}u^k=\widetilde{\mathcal{T}}w^k$. Furthermore, since $\|\widetilde{\mathcal{T}}w^k - w^k\|\to 0$ (cf. \cite[Corollary 5.17 (ii)]{BCombettes}), we get $w^*=\mathcal{C}^*u^{**}$. Thus, $u^*=(\mathcal{M}+\mathcal{A})^{-1}\mathcal{C}w^*=\mathcal{T}u^{**}=u^{**}$. 
\end{proof}
\begin{remark}\label{remark:quotient} If we restrict $\mathcal{C}^*$ to $\Image\mathcal{M} = \ker \mathcal{M}^{\perp}$, since by \cref{CholeskyDecomposition} the operator $\mathcal{C}^*$ is onto and $\ker \mathcal{C}^*=\ker\mathcal{M}$, the resulting map is a bijection between $\Image\mathcal{M}$ and $\bD$. Furthermore, if on $\Image\mathcal{M}$ we consider the Hilbert space structure given by the $\mathcal{M}$-scalar product, then it is easy to check that $\mathcal{C}^*|_\mathcal{\Image\mathcal{M}}:\Image\mathcal{M}\to \bD$ is actually an isometric isomorphism. For this reason, when $\mathcal{M}$ is not onto, $\bD$ could be thought as a strictly smaller space than $\bH$.
\end{remark}

\subsection{Linear convergence}\label{sec:LinearConvergence}
Let $\mathcal{A}: \bH\to 2^{\bH}$ be an operator and $\mathcal{M}$ be an admissible preconditioner with closed range, and let $\mathcal{M}=\mathcal{C}\mathcal{C}^*$ be a decomposition of $\mathcal{M}$ according to \cref{CholeskyDecomposition}. Having at hand an explicit characterization of the reduced algorithm, we may wonder which conditions on $\mathcal{A}$ should be imposed to have a strongly monotone reduced operator, as this would imply that $\widetilde{\mathcal{T}}$ is a Banach contraction \cite[Proposition 23.13]{BCombettes}, and we could therefore conclude the linear convergence of the reduced algorithm, at least if $\lambda_k$ is constant. To this end, as it will become clearer afterwards, we introduce the analogue of the strong-monotonicity (see \cite[Definition 22.1]{BCombettes}) adapted in our degenerate case, that is:
\begin{defi}[$\mathcal{M}$-strong-monotonicity] Let $\mathcal{M}:\bH\to\bH$ be a bounded linear positive semi-definite operator. Then $\mathcal{B}:\bH\to 2^{\bH}$ is $\mathcal{M}$-$\alpha$-strongly monotone if
        \begin{equation}\label{eq:AtildeCocoercive}
        \langle v-v',u-u'\rangle_{\mathcal{M}}\geq \alpha \|u-u'\|_{\mathcal{M}}^2 \quad \text{for all} \ (u,v),(u',v')\in \mathcal{B}.
        \end{equation}    
\end{defi}

As in \cref{prop:MMaximalMonotonicityProp} we have a characterization of the $\mathcal{M}$-strong monotonicity of $\mathcal{M}^{-1}\mathcal{A}$ in terms of a weaker notion of strong monotonicity for $\mathcal{A}\cap(\bH\times \Image\mathcal{M})$, which will be useful in applications.
\begin{prop}
Let $\mathcal{A}:\bH\to 2^{\bH}$ be an operator and $\mathcal{M}:\bH\to \bH$ be a positive semi-definite operator, then $\mathcal{M}^{-1}\mathcal{A}$ is $\mathcal{M}$-$\alpha$-strongly monotone if and only if 
\begin{equation}\label{eq:AtildeCocoercive2} 
\langle v-v',u-u'\rangle \geq \alpha \|u-u'\|^2_{\mathcal{M}}\quad \text{for all} \ (u,v),(u',v')\in\mathcal{A}\cap( \bH\times\Image \mathcal{M}).
\end{equation}
\end{prop}
\begin{proof}
    Using \cref{M-1AeA}, we have the following equivalences
    \begin{align*}
    \langle v-v',&u-u'\rangle \geq \alpha \|u-u'\|^2_{\mathcal{M}}\quad \text{for all} \  (u,v),(u',v')\in\mathcal{A}\cap( \bH\times\Image \mathcal{M})\\
    \iff &\langle \mathcal{M}\tilde{v}-\mathcal{M}\tilde{v}',u-u'\rangle \geq \alpha \|u-u'\|^2_{\mathcal{M}}\quad \text{for all} \  (u,\tilde{v}),(u',\tilde{v}')\in\mathcal{M}^{-1}\mathcal{A}\\
    \iff &\langle \tilde{v}-\tilde{v}',u-u'\rangle_\mathcal{M} \geq \alpha \|u-u'\|^2_{\mathcal{M}}\quad \text{for all} \  (u,\tilde{v}),(u',\tilde{v}')\in\mathcal{M}^{-1}\mathcal{A},
    \end{align*}
    which yields the thesis.
\end{proof}
The following theorem gives the right conditions we are looking for.
\begin{theorem}\label{StrongConvergence}
 Let $\mathcal{A}:\bH\to 2^{\bH}$ be an operator and $\mathcal{M}$ be an admissible preconditioner. Then $\mathcal{M}^{-1}\mathcal{A}$ is $\mathcal{M}$-$\alpha$-strongly monotone with $\alpha>0$ if and only if $\mathcal{C}^*\rhd\mathcal{A}$ is $\alpha$-strongly monotone. In that case, $\widetilde{\mathcal{T}}$ is a contraction of factor $1/(1+\alpha)$.
\end{theorem}
\begin{proof}
Let $(w_1,w_1'),\ (w_2,w_2')\in \mathcal{C}^*\rhd \mathcal{A}$. By \cref{MonotoneOperatorsPushForward} there exist $(u_1,u_1')$, $(u_2, u_2') \in \mathcal{M}^{-1}\mathcal{A}$ such that $(\mathcal{C}^*u_i, \mathcal{C}^*u_i')=(w_i,w_i')$ for $i =1,2$. Thus, using the $\mathcal{M}$-strong monotonicity of $\mathcal{M}^{-1}\mathcal{A}$, we have
\begin{align*}
    \langle w_1'-w_2',w_1-w_2\rangle &= \langle \mathcal{C}^*u_1'-\mathcal{C}^*u_2',\mathcal{C}^*u_1-\mathcal{C}^*u_2\rangle = \langle u_1'-u_2',u_1-u_2\rangle_\mathcal{M}\\
    & \geq \alpha \|u_1-u_2\|_\mathcal{M}^2=\alpha \|\mathcal{C}^*u_1-\mathcal{C}^*u_2\|^2= \alpha \|w_1-w_2\|^2.
\end{align*}
By \cite[Proposition 23.13]{BCombettes}, $\widetilde{\mathcal{T}}$ is a Banach contraction with constant $1/(1+\alpha)$.
\end{proof}
\begin{remark}
    It is clear that condition \cref{eq:AtildeCocoercive}, which is equivalent to \cref{eq:AtildeCocoercive2}, is in general weaker than the strong monotonicity of $\mathcal{A}\cap (\bH\times \Image\mathcal{M})$. We stress that \cref{eq:AtildeCocoercive2} is also quite easy to check in some practical cases, see \cref{sec:applications}.
\end{remark}
\begin{corollary}
    If $\mathcal{A}$ is $\alpha \|\mathcal{C}\|^2$-strongly monotone and $\zer \mathcal{A} \neq \emptyset$, then $\widetilde{\mathcal{T}}$ is a contraction of factor $1/(1+\alpha)$.
\end{corollary}
\begin{proof} We have for all $(u,v), \ (u',v')\in\mathcal{A}$ that
\begin{equation} \langle v-v',u-u'\rangle \geq \alpha \|\mathcal{C}\|^2 \|u-u'\|^2\geq \alpha \|\mathcal{C}^*(u-u')\|^2_{\mathcal{M}}=\alpha \|u-u'\|^2_{\mathcal{M}}, \end{equation}
    which implies \cref{eq:AtildeCocoercive}.
\end{proof}

\section{Application to splitting algorithms}\label{sec:applications}
Now we consider problems of the form \cref{sum}, i.e.,~an inclusion for the sum of $(N+1)$ maximal monotone operators $A_{i}$ and we assume that each $A_i$ is \textit{simple}, i.e.,~we are able to cheaply compute the resolvent $J_{A_i}$ for each $i \in \{0,\dots,N\}$. The strategy we focus on consists in reformulating \cref{sum} as an inclusion problem \cref{monotoneinclusionIntro} in a proper product space and applying a degenerate preconditioned proximal point algorithm.

As we illustrate now, there are many ways to broadcast problem \cref{sum} as the inclusion problem \cref{monotoneinclusionIntro}, generating a wide range of (known and new) splitting algorithms.

\paragraph{Chambolle-Pock} As firstly noticed in \cite{HeYuan}, one of the most popular splitting algorithm that can be seen as a \textsf{PPP} method is the so-called \textit{Chambolle-Pock} (\textsf{CP}) scheme, which aims to solve the following problem
\begin{equation}\label{convexproblem}
    \min_{x\in \mathcal{H}} f(x)+g(Lx)
\end{equation}
where $f:\mathcal{H}\to \mathbb{R}\cup \{+\infty\},~g:\mathcal{K}\to \mathbb{R}\cup \{+\infty\}$ are convex, lower semicontinuous functions, $\mathcal{H}$ and $\mathcal{K}$ are real Hilbert spaces and $L\in \mathcal{B}(\mathcal{K},\mathcal{H})$. From a monotone operator standpoint, problem \cref{convexproblem} can be formulated as
\begin{equation}\label{ChambollePock}
\text{find x} \ \in \mathcal{H} \ \text{such that:} \ 0 \in (A+L^*BL)x
\end{equation}
where $A=\partial f$ and $B=\partial g$ are maximal monotone operators on $\mathcal{H}$ and $\mathcal{K}$ respectively. Following \cite{HeYuan} we reformulate this problem in $\bH:=\mathcal{H}\times \mathcal{K}$ introducing $\mathcal{A}$ and $\mathcal{M}$ defined by
\begin{equation*}
\mathcal{A}:= 
\begin{bmatrix}
A & L^* \\
-L & B^{-1}
\end{bmatrix},
\quad \mathcal{M}:= 
\begin{bmatrix}
\frac{1}{\tau}I & -L^* \\
-L & \frac{1}{\sigma}I
\end{bmatrix}.
\end{equation*}
 Let $u:=(x,y)\in \bH$, the inclusion problem $0 \in \mathcal{A}u$ in $\bH$ is indeed equivalent to \cref{ChambollePock}. The \textsf{PPP} iteration \cref{eq:PPP} with $\lambda_k=1$ and starting point $(x^0,y^0)\in \bH$ hence writes as
\begin{equation}\label{ChambollePockAlg}
  \begin{cases}
    x^{k+1} = J_{\tau A}( x^k-\tau L^* y^{k}),\\
    y^{k+1} = J_{\sigma B^{-1}}\left(y^k + \sigma L(2x^{k+1}-x^k)\right).
  \end{cases}
\end{equation}
We focus on the degenerate case $\tau\sigma \|L\|^2=1$ (since the case $\tau\sigma \|L\|^2<1$ induces a positive definite preconditioner). The operator $\mathcal{A}$ is maximal monotone on $\bH$ since it is the sum of the maximal monotone operator $(x,y)\mapsto Ax \times B^{-1}y$
(which is maximal monotone since $A$ and $B^{-1}$ are maximal monotone \cite[Proposition 20.23]{BCombettes}) and the operator $(x,y)\mapsto (L^* y, -Lx)$ (which is maximal monotone, since it is skew symmetric, linear and bounded). Additionally $(\mathcal{M}+\mathcal{A})^{-1}$ is Lipschitz, since
\begin{equation}\label{ChambollePockG}
 \left(\mathcal{M}+\mathcal{A}\right)^{-1}
 : \ (u_1, u_2) \mapsto
 \left(J_{\tau A}(\tau u_1), J_{\sigma B^{-1}}(2\sigma LJ_{\tau A}(\tau u_1)+ \sigma u_2)\right).
\end{equation}
which is a composition of Lipschitz functions. The Lipschitz constant is not relevant in this context, in fact, we can already conclude using \cref{PPPwithmaxmon} that even if $\tau\sigma\|L\|^2=1$, the \textsf{CP} algorithm, i.e.,~iteration \cref{ChambollePockAlg}, converges weakly to a solution as soon as \cref{ChambollePock} has a solution. This result has been proven in \cite{OConnor2020OnTE} in a finite dimensional setting and more recently in \cite[Section 5]{condat:hal-00609728}, reformulating the so-called \textit{Generalized} \textsf{CP} (which actually includes \textsf{CP} as a particular case) as an instance of \textsf{DRS} in a modified space and relying on existing proofs of convergence for \textsf{DRS}, namely \cite[Theorem 26.11]{BCombettes}, that exploits the particular structure of the problem.
\begin{remark}
	In this case, the reduction of variables in the sense of \cref{ReducedAlg} could not easily be established in general. Indeed, in an infinite dimensional setting it is not clear if $\mathcal{M}$ has necessarily closed range or not. Moreover, the decomposition of $\mathcal{M}$ according to \cref{CholeskyDecomposition} is not explicit and the space $\bD$ is not necessarily much smaller than $\bH$.
\end{remark}

\paragraph{Douglas-Rachford}    Let $A$ and $B$ be two maximal monotone operators on $\mathcal{H}$. The \textsf{DRS} method can be introduced as a particular case of \textsf{CP} with $L=I$ and stepsize $\tau=\sigma=1$ for solving
\begin{equation}\label{sumtwo}
\text{find} \ x \in \mathcal{H} \ \text{such that:} \ 0 \in Ax+Bx
\end{equation}
scaled by $\sigma > 0$. As anticipated in the introduction, \textsf{DRS} admits a genuinely degenerate \textsf{PPP} formulation on $\bH=\mathcal{H}^2$, with $\mathcal{M},~\mathcal{A}$ according to \cref{DouglasRachfordSettingIntro}, and \textsf{PPP} iteration \cref{eq:PPP} with $\lambda_k=1$ according to \cref{intro:DRSasPPP}.

In this case, the decomposition of $\mathcal{M}$ according to \cref{CholeskyDecomposition} is explicit, we have $\mathcal{C}^*=\begin{bmatrix}
I & -I
\end{bmatrix}$, with $\bD=\mathcal{H}$, which gives exactly the same substitution we performed in the introduction. Since $\mathcal{M}$ has closed range, this leads us to the reduced algorithm
\begin{equation}\label{DRS}
w^{k+1}:=\left(I+\mathcal{C}^*\rhd \mathcal{A}\right)^{-1}w^k=w^k+J_{\sigma B}(2J_{\sigma A}w^k-w^k)-J_{\sigma A}w^k,
\end{equation}
that is \textsf{DRS} according to \cref{intro:DRSasRPPP}. Notice that we only need to store one variable (instead of two).

\textsf{DRS} is known to be a proximal point algorithm since the work of Eckstein and Bertsekas \cite{EcksteinBertsekas}, but no link between the operator $\mathcal{S}_{A,B}$ from \cite{EcksteinBertsekas} and the one in \cref{DouglasRachfordSettingIntro} has been observed so far. We have, in fact, that $\mathcal{S}_{A,B}$ characterized in \cite{EcksteinBertsekas} coincides exactly with $\mathcal{C}^*\rhd \mathcal{A}$. 

While the convergence of the reduced method \cref{DRS}, in case \cref{sumtwo} has a solution, has been established in the seminal paper \cite{Lions1979SplittingAF} already, the convergence of the non-reduced method (in our notation: the convergence of the \textsf{PPP} method \cref{eq:PPP} for $\mathcal{M}$ and $\mathcal{A}$ according to \cref{DouglasRachfordSettingIntro}), and in particular the convergence of the sequence $\{x^k\}_k,~x^k = J_{\sigma A}w^k$, has been open for a longer time and has been settled more recently in~\cite{SVAITER2011} and, afterwards in \cite{Davis2017} for a more general framework. Using the tools we developed so far, we can conclude the same results in an arguably easier way following from our general framework on degenerate and reduced proximal point iterations. In fact (since we are in a particular case of the \textsf{CP} algorithm), $\mathcal{A}$ is maximal monotone and $(\mathcal{M}+\mathcal{A})^{-1}$ is Lipschitz, and by \cref{remark:samesolution}, we can conclude that the \textsf{PPP} iterates generated by \cref{DouglasRachfordSettingIntro} converge weakly in $\bH$ to a zero of $\mathcal{A}$. In particular, if $w^k$ is generated by \textsf{DRS}, then $x^k:=J_{\sigma A}w^k$ converges weakly in $\mathcal{H}$ to a solution of \cref{sumtwo}.

Exploiting our multivalued operator framework, in particular \cref{StrongConvergence}, we can also shed light on recent results about the strong convergence of \textsf{DRS} \cite{Pontus}. 
Introducing the notion of cocoercivity as in \cite[Definition 4.10]{BCombettes}, we can obtain the following theorem, where, for each case we can prove that $
\mathcal{M}^{-1}\mathcal{A}$ is $\mathcal{M}\text{-}\alpha$-strongly monotone. The proof is postponed to the appendix.
\begin{theorem}[Linear convergence]\label{thm:DRS-strongconvergence}
	Let $\mathcal{A},~\mathcal{M}$ be the operators defined in \cref{DouglasRachfordSettingIntro}, assume that $\zer \mathcal{A}\neq \emptyset$ and that one of the following holds for $\mu,\beta>0$:
	\begin{enumerate}
		\item  $A$ $\mu$-strongly monotone and $B$ $1/\beta$-cocoercive;
		\item $A$ $\mu$-strongly monotone and $1/\beta$-cocoercive;
		\item $A$ $\mu$-strongly monotone and $\beta$-Lipschitz continuous,
	\end{enumerate}
	or the analogous statements with $A$ and $B$ inverted. Then, the reduced algorithm, i.e.,~\textsf{DRS}, converges strongly with linear rate $r=1/(1+\alpha)$ where
	\begin{equation*}
	1. \    \alpha=\min\bigg\{\frac{\sigma\mu}{2},\frac{1}{2\sigma \beta}\bigg\}, \quad 2. \  \alpha=\frac{\sigma \mu}{\sigma^2\mu\beta+1}, \quad 3. \ \alpha=\frac{\sigma\mu}{\sigma^2\beta^2+1}.
	\end{equation*}
\end{theorem}

\begin{remark}
	Giving tight rates is beyond the scope of this paper, but the rates we derived with our techniques are already satisfying. The rate for case 1 could be optimized taking $\sigma=1/\sqrt{\mu \beta}$ which gives the rate $\sqrt{\beta/\mu}/(\sqrt{\beta/\mu}+1/2)$. This is the same obtained in \cite[Theorem 5.6]{Pontus} if with the notation of \cite{Pontus} we put $\alpha = 1/2$ and optimize only on $\gamma$. Regarding case 2 and 3, the rates we derived are (in practice) almost the same as those given in \cite[Theorem 6.5]{Pontus}, but not tight. 
\end{remark}

\paragraph{Peaceman-Rachford and overrelaxed generalizations} In this section we show that relaxation parameters for \textsf{DRS} can be encoded from a purely monotone operator point of view. 
Let $\sigma >0,~\alpha > -1$ and consider the following operators
\begin{equation}\label{eq:OverelaxedDRS}
\mathcal{A}_{\alpha}:=
\begin{bmatrix}
\alpha I+\sigma A_0 & -I & -I \\
I& 0 & -I \\
(1-2\alpha) I & I & \alpha I+\sigma A_1
\end{bmatrix},\quad
\mathcal{M}:=
\begin{bmatrix}
I & I & I \\
I & I & I \\
I & I & I \\
\end{bmatrix}.
\end{equation}
Let $u=(x_0,v,x_1)\in \bH:=\mathcal{H}^3$, note that $0\in\mathcal{A}_\alpha u$ is equivalent to $v+x_1-\alpha x_0 \in \sigma A_0 x_0$,~ $x_0=x_1$ and $(2\alpha-1)x_0-v-\alpha x_1\in \sigma A_1 x_1$, which implies $0\in A_0 x_0 +A_1 x_0$. Conversely, if $0\in A_0 x +A_1 x$, setting $x_0=x_1=x$, there exists $a_0\in \sigma A_0 x$ and $a_1\in \sigma A_1 x$ such that $0=a_0+a_1$. Setting $v=(\alpha-1)x_0+a_0$ yields $0\in \mathcal{A}_\alpha u$. Hence, finding zeros of $\mathcal{A}_\alpha$ is in a certain sense equivalent to finding zeros of $A_0+A_1$. 

The operator $\mathcal{M}$ is an admissible preconditioner to $\mathcal{A}_\alpha$ (whenever $\alpha>-1$) and the iteration updates can be computed explicitly. Indeed, we need to invert $(\mathcal{M}+\mathcal{A}_\alpha)$ which has a lower triangular structure. This inversion only involves the resolvents $J_{\gamma A_0}$ and $J_{\gamma A_1}$, where $\gamma=\sigma/(1+\alpha)$, as well as linear combinations and is hence Lipschitz continuous. Let $w=\mathcal{C}^*u=x_0+v+x_1$, where $\mathcal{C}^*=\begin{bmatrix}
I & I & I
\end{bmatrix}$ is a decomposition of $\mathcal{M}$ according to \cref{CholeskyDecomposition}, then $\mathcal{T}_\alpha=(\mathcal{M}+\mathcal{A}_\alpha)^{-1}\mathcal{M}$ can be expressed as 
\begin{equation*}
\mathcal{T}_\alpha u = \left( J_{\gamma A_0}\left(\frac{w}{1+\alpha}\right), \ w-2J_{\gamma A_0}\left(\frac{w}{1+\alpha}\right) , \ J_{\gamma A_1}\left(2J_{\gamma A_0}\left(\frac{w}{1+\alpha}\right)-\frac{w}{1+\alpha}\right) \right).
\end{equation*}
Taking as relaxation parameters $\theta_k=\lambda_k/(1+\alpha)$ for all $k \in \mathbb{N}$, and scaling $w$ with a factor of $(1+\alpha)^{-1}$, i.e.,~considering $\widetilde{w}=w/(1+\alpha)$ instead of $w$ in the iteration updates $w^{k+1}=w^k+\lambda_k(\widetilde{\mathcal{T}}_\alpha w^k-w^k)$ (where $\widetilde{\mathcal{T}}_\alpha=(I+\mathcal{C}^*\rhd \mathcal{A}_\alpha)^{-1}$) we get
\begin{equation}\label{relaxedDRS}
  \begin{cases}
    x_0^{k+1} = J_{\gamma A_0}\widetilde{w}^{k},\\
    x_1^{k+1} = J_{\gamma A_1}\left(2x_0^{k+1}-\widetilde{w}^{k}\right),\\
    \widetilde{w}^{k+1} =\widetilde{w}^k+\theta_k(x_1^{k+1}-x_0^{k+1}).
  \end{cases}
\end{equation}
We are interested in \textit{overrelaxation} cases, i.e.,~with $\theta_k\geq 2$ $\big(\theta_k=2$ for all $k$ is the so-called Peaceman-Rachford splitting$\big)$, that correspond to a non-positive $\alpha$. In these cases we cannot expect to have an \textit{unconditionally} convergent algorithm. Indeed, denoting by $a_i\in A_i x_i, \  a_i'\in A_i x_i'$, $\Delta a_i=a_i-a_i'$ and $\Delta x_i=x_i-x_i'$, for $i=0,1$, after some computations we have (with the usual mild abuse of notation)
\begin{equation}\label{eq:OvverelaxedMonotonicity}
\langle \mathcal{A}_\alpha u-\mathcal{A}_\alpha u',u-u'\rangle = \alpha\|\Delta x_1-\Delta x_2\|^2+\sigma\langle  \Delta a_0, \Delta x_1\rangle +\sigma \langle \Delta a_1, \Delta x_2\rangle,
\end{equation}
which could be negative in general if $\alpha$ is negative. However, it is clear from \cref{eq:OvverelaxedMonotonicity} that in that case to gain a non-negative expression, suitable additional assumptions on $A_0$ and $A_1$ should be imposed in order to achieve monotonicity of $\mathcal{A}_\alpha$.   
\begin{prop}\label{prop:OverelaxedDRS}
	\label{prop:overelaxed}
	Let $A_0$ and $A_1$ be maximal monotone operators such that for $i \in \{0,1\}$, $A_i$ is $\mu_i$-strongly monotone and assume that $\zer(A_0+A_1)\neq \emptyset$. For $\alpha >-1$ and $\sigma>0$ set $\gamma=\sigma/(1+\alpha)$ and assume that $\{\theta_k\}_k$ satisfy
	\begin{equation}\label{eq:ovverelaxed_upperparameter}
		 \theta_k \in \left[0,2+\frac{2\gamma\mu_0\mu_1}{\mu_0+\mu_1}\right], \quad \sum_k \theta_k\left(2+\frac{2\gamma\mu_0\mu_1}{\mu_0+\mu_1}-\theta_k\right)= +\infty.
	\end{equation}
	Then, the sequence $\{\widetilde{w}^k\}_k$ generated by \cref{relaxedDRS} with starting point $\widetilde{w}^0$ weakly converges to a fixed point $\widetilde{w}^*$ such that $x^*=J_{\gamma A_0}\widetilde{w}^*$ is a solution of $0\in A_0x+A_1x$. Moreover, the sequences $\{x_0^k\}_k$, $\{x_1^k\}_k$ also converge weakly to $x^*$.
\end{prop}

\begin{proof}
    We rely on \cref{remark:samesolution}. The Lipschitz regularity of $(\mathcal{M}+\mathcal{A}_\alpha)^{-1}$ is clear. As for monotonicity of $\mathcal{A}_\alpha$,  using the strong monotonicity of $A_0$ and $A_1$ we have that
        \begin{equation*}
           \cref{eq:OvverelaxedMonotonicity}  \geq (\alpha+\sigma \mu_0)\|\Delta x_0\|^2+(\alpha+\sigma \mu_1)\|\Delta x_1\|^2-2\alpha\langle\Delta x_0,\Delta x_1\rangle.
        \end{equation*}
    The right-hand side is non-negative if $\alpha+\sigma \mu_0\geq 0$, $\alpha+\sigma \mu_1 \geq 0$, and $(\alpha+\sigma\mu_0)(\alpha+\sigma\mu_1)=\alpha^2$, i.e.,~$\alpha=-\frac{\sigma\mu_0\mu_1}{\mu_0+\mu_1}$. Using that $\gamma=\sigma/(1+\alpha)$ then yields $\alpha= -\frac{\gamma\mu_0\mu_1}{\gamma\mu_0\mu_1+\mu_0+\mu_1}$. Now, since $\lambda_k = (1+\alpha)\theta_k$, condition \cref{eq:ovverelaxed_upperparameter} leads to $\lambda_k \in [0,2]$, $\sum_k\lambda_k(2-\lambda_k)= +\infty$. The maximality of $\mathcal{A}_\alpha$ follows from standard arguments, e.g., \cite[(i) Corollary 25.5]{BCombettes}.
\end{proof}

\cref{prop:OverelaxedDRS} is known, and has been discussed in \cite{DONG20104307} and \cite{Monteiro2018} with $\mu_0=\mu_1$ and $\theta_k=\theta$ fixed for all $k \in \mathbb{N}$. 

\paragraph{Forward Douglas-Rachford} In this example we show that even forward steps can be encoded in a \textsf{PPP} formulation. Consider the 3-operator problem
\begin{equation}\label{3op_pbl}
\text{find} \ x \in \mathcal{H} \ \text{such that:} \ 0 \in A_0x+A_1x+Cx,
\end{equation}
where $A_0, \ A_1$ and $C$ are three maximal monotone operators and $C$ has full domain and is $1/\beta$-cocoercive, with $\beta >0$. The so-called Forward Douglas-Rachford \textsf{(FDR)} algorithm, also known as Davis-Yin algorithm \cite{Davis2017}, can be thought as a \textsf{DRS} scheme with an additional forward term. Whether it can be seen from a degenerate preconditioned point of view is not obvious at first glance, but here is a proper splitting
\begin{equation}\label{DavisYinSplitting}
\mathcal{A}_\alpha:= 
\begin{bmatrix}
\alpha I+\sigma A_0 & -I & -I \\
I& 0 & -I \\
(1-2\alpha)I + \sigma C & I & \alpha I+\sigma A_1 
\end{bmatrix}, \quad
\mathcal{M}:=
\begin{bmatrix}
I & I & I \\
I & I & I \\
I & I & I \\
\end{bmatrix}.
\end{equation}
As can be seen analogous to the previous case, denoting by $u=(x_0, v, x_1)\in \bH:=\mathcal{H}^3$, $x$ solves \cref{3op_pbl} if and only if $0\in \mathcal{A}_\alpha u$ with $x_0=x_1=x$ and $v=(\alpha-1)x-a_0-C x$ for $a_0\in \sigma A_0 x$, $a_1\in \sigma A_1 x$ such that $0=a_0+a_1+\sigma Cx$.
Furthermore, whenever $\alpha>-1$, the operator $\mathcal{M}$ is an admissible preconditioner to $\mathcal{A}_\alpha$, has closed range, block-rank one and obeys $\mathcal{M}=\mathcal{C}\mathcal{C}^*$ for $\mathcal{C}^*=\begin{bmatrix}
I & I & I
\end{bmatrix}$. Also in this case, $(\mathcal{M}+\mathcal{A}_\alpha)^{-1}$ is Lipschitz continuous. Performing the same sort of computations as in the paragraph devoted to the analysis of the Peaceman-Rachford splitting, in particular putting $\gamma=\sigma/(1+\alpha)$, we eventually get to the following resolvent $\mathcal{T}_\alpha := (I+\mathcal{M}^{-1}\mathcal{A}_\alpha)^{-1}$
\begin{align*}
    \mathcal{T}_\alpha u = \bigg( & J_{\gamma A_0}\left(\frac{w}{1+\alpha}\right), \ w-2J_{\gamma A_0}\left(\frac{w}{1+\alpha}\right),\\ & J_{\gamma A_1}\left(2J_{\gamma A_0}\left(\frac{w}{1+\alpha}\right)-\frac{w}{1+\alpha}-\gamma C J_{\gamma A_0}\left(\frac{w}{1+\alpha}\right)\right) \bigg) ,
\end{align*}
where $w=\mathcal{C}^*u=x_0+v+x_1$. Again, introducing new relaxation parameters $\theta_k = \lambda_k/(1+\alpha)$ for all $k \in \mathbb{N}$, and scaling $w$ with a factor of $(1+\alpha)^{-1}$, i.e.,~considering $\widetilde{w}=w/(1+\alpha)$ instead of $w$, we obtain a reduced algorithm of the form
\begin{equation*}
  \begin{cases}
    x_0^{k+1} = J_{\gamma A_0}\widetilde{w}^{k},\\
    x_1^{k+1} = J_{\gamma A_1}(2x_0^{k+1}-\widetilde{w}^{k}-\gamma Cx_0^{k+1}),\\
   \widetilde{w}^{k+1}  =\widetilde{w}^k+\theta_k(x_1^{k+1}-x_0^{k+1}),
  \end{cases}
\end{equation*}
which coincides with \textsf{FDR} with relaxation parameters $\theta_k$.  

The monotonicity of $\mathcal{A}_\alpha$ in this case is not obvious and fails in general. This is not surprising, as the \textsf{FDR} operator is in general only averaged but not firmly non-expansive, as also noticed in \cite{Davis2017}. However, assuming suitable regularity for $C$, the proposed \textsf{PPP} framework provides straightforward conditions on the parameters that ensure convergence. Indeed, introducing similar notations to those introduced in \cref{thm:DRS-strongconvergence} and \cref{prop:overelaxed} (namely, $\Delta C x_0= C x_0- C x_0'$ and $\Delta x_i=x_i-x_i'$, $i=0,1$), we have
\begin{equation*}
\begin{aligned}
\langle \mathcal{A}_{\alpha}u-\mathcal{A}_{\alpha}u',u-u'\rangle & \geq \alpha \|\Delta x_0-\Delta x_1\|^2+ \sigma \langle \Delta C x_0, \Delta x_1 \rangle\\
& =\alpha \|\Delta x_0-\Delta x_1\|^2 + \sigma \langle \Delta C x_0, \Delta x_1 - \Delta x_0 \rangle + \sigma \langle \Delta C x_0, \Delta x_0 \rangle \\
& \geq \alpha \|\Delta x_0-\Delta x_1\|^2 + \sigma \langle \Delta C x_0, \Delta x_1 - \Delta x_0 \rangle + \frac{\sigma}{\beta}\| \Delta C x_0 \|^2.
\end{aligned}
\end{equation*}
Thus, we obtain the desired monotonicity by setting $\alpha= \frac{\sigma\beta}{4}>0$. In terms of $\gamma = \sigma/(1+\alpha)$ this means $\alpha=\frac{\gamma\beta}{4-\gamma\beta}$, which is positive for all $\gamma \in (0,4/\beta)$. Since we want to use \cref{remark:samesolution}, this choice of $\alpha$ yields the conditions $\theta_k\in\big[0,2-\frac{\gamma\beta}{2}\big]$, $\sum_k\theta_k(2-\frac{\gamma\beta}{2}-\theta_k)=+\infty$, which slightly improves those found in \cite[Theorem 2.1]{Davis2017}. The maximality of $\mathcal{A}_\alpha$ follows from \cite[Corollary 25.5]{BCombettes} as it can be seen as the sum of the maximal monotone operator $\mathrm{diag}(\sigma A_0, 0, \sigma A_1)$ and a maximal monotone operator with full domain. Using \cref{remark:samesolution} we get that in case \cref{3op_pbl} has a solution, for all starting points $\widetilde{w}^0 \in \mathcal{H}$, the sequence $\{\widetilde{w}^k\}_k$ generated by \textsf{rPPP} weakly converges to some $\widetilde{w}^*$ such that $J_{\gamma A_0}\widetilde{w}^*$ is a solution of \cref{3op_pbl}, moreover, also the sequences $\{x_0^k\}_k$, $\{x_1^k\}_k$ weakly converge to this same solution.

\begin{remark}\label{rem:overrelaxedFDR}
	It is also possible to consider the case $\alpha\leq 0$, that would correspond to a Peaceman-Rachford and an overrelaxed version of \textsf{FDR}, but further assumptions on the operators (e.g.,~$A_0$ and $A_1$ strongly monotone) would be required.
\end{remark}

\subsection{Parallel and sequential generalizations of \textsf{FDR}}

We now turn our attention to the $(N+1)$-operator case. We show how the proposed framework provides different new generalizations of the above-discussed splitting methods to tackle a larger class of operators. In particular, we focus on \textsf{FDR} \cref{DavisYinSplitting} (and hence, on \textsf{DRS} and the well-known Forward-Backward and Backward-Forward methods as particular cases, see \cite{Davis2017}), we describe its parallel version from a \textsf{PPP} standpoint and introduce a new \textit{sequential} generalisation. Consider the following problem:
\begin{equation}\label{sumDY}
\text{find} \ x \in \mathcal{H} \ \text{such that:} \ 0 \in A_0 x+\sum_{i=1}^N \big[A_ix +C_i x\big],
\end{equation}
where $A_0, \ A_i, \ C_i$ for $i \in \{1,\dots, N\}$ are maximal monotone operators and $C_i$ for $i \in \{1,\dots, N\}$ have full domain and are $1/\beta-$cocoercive, with $\beta>0$.

\subsubsection{Parallel generalization of \textsf{FDR}}\label{sec:ParallelFDR}
One of the most popular approaches to extend a splitting algorithm to tackle \cref{sumDY} is to re-formulate the problem as a 3-operator problem on a suitable product-space, and then applying the 3-operator splitting to this larger problem, typically yielding schemes which are intrinsically parallel in nature. This is often referred to as \textit{product space trick}, or \textit{consensus technique} and is also known as \textit{Pierra's reformulation} \cite{Pierra1984}. We refer to \cite[Section 8]{Condat2020ProximalSA} and the references therein for further details. We show that the same scheme can be re-obtained as \emph{one} natural generalization of \cref{DavisYinSplitting} to tackle the $(2N+1)$-operator case \cref{sumDY}. Consider, for instance

\vspace{-0.4cm}
{\footnotesize
	\begin{equation*}\label{ParallelDavisYinSplitting}
	\arraycolsep=1.5pt
	\renewcommand{\arraystretch}{0.9}
	\mathcal{A}_\alpha:= 
	\begin{bmatrix}
	N\alpha I+\sigma A_0 & -I & -I & \dots & -I & -I\\
	I &  & -I & & & \\
	(1-2\alpha)I + \sigma C_1 & I & \alpha I+\sigma A_1  & & &\\
	\vdots & & & \ddots & & \\
	I &  &  & & & -I\\
	(1-2\alpha)I + \sigma C_N &  &  & &  I  & \alpha I+\sigma A_N
	\end{bmatrix}, \quad
	\arraycolsep=3pt
	\mathcal{M}:=
	\begin{bmatrix}
	NI & I & I & \dots & I & I\\
	I & I & I &  & &  \\
	I & I & I &  & &\\
	\vdots & & & \ddots & &\\
	I &  &  & & I & I \\
	I &  &  & & I & I
    \end{bmatrix}.
	\end{equation*}
}%
It is easy to see that if $u:=(x_0, v_1, x_1, \dots ,v_N,x_N)\in \bH= \mathcal{H}^{(2N+1)}$ satisfies $0\in\mathcal{A}_\alpha u$ then $x =x_0=\dots =x_N$ solves \cref{sumDY}, and that a solution $x$ of \cref{sumDY} yields $0\in \mathcal{A}_\alpha u$ setting $x_0=\dots=x_N=x$ and $v_i=(\alpha-1)x-a_i-\sigma C_i x$ for all $i\in\{1,\dots,N\}$, where $a_i\in \sigma A_i x$, $i\in\{0,\dots,N\}$, such that $0=a_0+\sum_{i=1}^N a_i + \sigma C_i x$. Also, choosing $\mathcal{C}^*:u\mapsto (x_0+v_1+x_1,~\dots~,~x_0+v_N+x_N)$ we obtain an onto decomposition $\mathcal{M}=\mathcal{C}\mathcal{C}^*$, showing that \textsf{rPPP} reduces the number of variables from $(2N+1)$ to $N$. Moreover, $(\mathcal{M}+\mathcal{A}_\alpha)$ is lower-triangular, hence easy to invert, with $(\mathcal{M}+\mathcal{A}_\alpha)^{-1}$ easily seen to be Lipschitz continuous. Denoting by $w=\mathcal{C}^*u \in \mathcal{H}^N$ we obtain $w_i = x_0+v_i+x_i$, for all $i\in\{1,\dots,N\}$, that we re-scale with a factor of $(1+\alpha)^{-1}$, namely considering $\widetilde{w}_i=w_i/(1+\alpha)$ instead of $w_i$ for $i\in\{1,\dots,N\}$. We also set $\gamma = \sigma/(1+\alpha)$. Thus, we can explicitly compute $\widetilde{\mathcal{T}}_\alpha=\left(I+\mathcal{C}^*\rhd \mathcal{A}_\alpha\right)^{-1}$  and the reduced algorithm with relaxation parameters $\theta_k=\lambda_k/(1+\alpha)$ writes as
\begin{equation}\label{parallelDY}
\left\{
\begin{aligned}
x_0^{k+1} & = J_{\frac{\gamma}{N} A_0}\left( \frac{1}{N}\sum_{i=1}^{N}\widetilde{w}_i^k\right),\\
x_{i}^{k+1} & =J_{\gamma A_i}(2x_0^{k+1}-\widetilde{w}_i^k-\gamma C_ix_0^{k+1}) & \textit{ for } i\in \{1,\dots, N\},\\
\widetilde{w}_i^{k+1} & = \widetilde{w}_i^k+\theta_k(x_{i}^{k+1}-x_0^{k+1}) & \textit{ for } i\in \{1,\dots, N\}.
\end{aligned}
\right.
\end{equation}
Notice that the reduced proximal point iteration in this case coincides exactly with the parallel extension of the so-called \textsf{FDR}/Davis-Yin algorithm presented, for instance in \cite{Condat2020ProximalSA}, for convex minimization. If $A_0=0$, then \textsf{rPPP} coincides with the Generalized Backward-Forward iteration \cite{raguet:GenFB} and if $C_i=0$ for all $i \in \{1,\dots, N-1\}$ we get the parallel extension of \textsf{DRS} that can be found in the literature for the convex minimization case \cite[Theorem 8.1]{Condat2020ProximalSA}. The novelty that the proposed framework provides is an alternative derivation of the scheme as a proximal point algorithm with respect to an explicit operator $\mathcal{C}^*\rhd \mathcal{A}_\alpha$, that allows to give a straightforward proof of convergence (below, see also \cite{Davis2017, raguet:GenFB} for similar versions) and to derive new variants (see next subsection).

\begin{theorem}\label{thm:parallelDY}
	Let $\widetilde{w}^k=(\widetilde{w}_1^k,\dots,\widetilde{w}_N^k)$ for $k \in \mathbb{N}$ be the sequence generated by the parallel \textsf{FDR} iteration in \cref{parallelDY} with starting point $\widetilde{w}^0\in \mathcal{H}^N$. Let $\gamma\in(0,4/\beta)$ and $\{\theta_k\}_k$ satisfy $\theta_k\in\big[0,2-\frac{\gamma\beta}{2}\big]$ and $\sum_k\theta_k\big(2-\frac{\gamma\beta}{2}-\theta_k\big)=+\infty$. If a solution to \cref{sumDY} exists, then $\{\widetilde{w}^k\}_k$ converges weakly to a $\widetilde{w}^*\in \mathcal{H}^N$ such that $J_{\frac{\gamma}{N} A_0}\left(\frac{1}{N}\sum_{i=1}^N\widetilde{w}_i^*\right)$ is a solution to \cref{sumDY}. Moreover, the sequences $\{x_i^k\}_k$, for $i \in \{0,\dots,N\}$, weakly converge to this solution.
\end{theorem}

\begin{proof}
	First, we need to check the monotonicity of $\mathcal{A}_\alpha$. Pick $(u,p)$ and $(u',p')\in \mathcal{A}_\alpha$. Then, denoting $\Delta x_0 = x_0-x_0'$, and $\Delta x_i = x_i-x_i'$, $\Delta C_i=C_ix_0-C_ix_0'$, $\Delta a_i=a_i-a_i'$ where $a_i\in A_ix_i$ and $a_i'\in A_ix_i'$ for all $i\in \{1,\dots, N\}$, we have
	\begin{align*}
	\langle p-p',& u-u'\rangle= N \alpha \|\Delta x_0\|^2+ \sigma\langle \Delta a_0,\Delta x_0\rangle\\
	&+ \sum_{i=1}^N -2\alpha\langle \Delta x_0,\Delta x_i\rangle+\sigma \langle \Delta C_ix_0,\Delta x_i\rangle +\alpha \|\Delta x_i\|^2+ \sigma \langle \Delta a_i,\Delta x_i\rangle\\
	& \geq \sum_{i=1}^N\alpha \|\Delta x_0\|^2 -2\alpha\langle \Delta x_0,\Delta x_i\rangle+\alpha \|\Delta x_i\|^2\\& \hspace{2.5cm}+\sigma \langle \Delta C_ix_0,\Delta x_i-\Delta x_0\rangle+\sigma \langle \Delta C_ix_0,\Delta x_0\rangle\\
	& \geq \sum_{i=1}^N \alpha \|\Delta x_0-\Delta x_i\|^2+\sigma \langle \Delta C_ix_0,\Delta x_i-\Delta x_0\rangle+\sigma /\beta \|\Delta C_ix_0\|^2,
	\end{align*}
	where the right-hand side is non-negative if we impose the condition $\alpha = \frac{\beta \sigma}{4} >0$. In terms of $\gamma = \sigma/(1+\alpha)$ this means $\alpha=\frac{\gamma\beta}{4-\gamma\beta}$, which is positive for all $\gamma \in (0,4/\beta)$. Notice moreover that $\lambda_k=\theta_k(1+\alpha)$ satisfy $\lambda_k\in[0,2]$ and $\sum_{k} \lambda_k(2-\lambda_k)=+\infty$. Thus, we are now under the hypothesis of \cref{ReductionAndConvergence}, that entails the convergence of $\{\widetilde{w}^k\}_k$ to a point $\widetilde{w}^*$ such that $(\mathcal{M}+\mathcal{A}_\alpha)^{-1}\mathcal{C}(1+\alpha)\widetilde{w}^* \in \zer \mathcal{A}_\alpha$. As a consequence, $J_{\gamma A_0}\left(\frac{1}{N}\sum_{i=1}^N\widetilde{w}_i^*\right)$ is a solution to \cref{sumDY}.
	
	As for the convergence of the associated \textsf{PPP}, we first observe that the maximality of $\mathcal{A}_\alpha$ is clear: indeed, $\mathcal{A}_\alpha$ can be written as a sum of a maximal monotone operator and a Lipschitz term \cite[Lemma 2.4]{brezis1973ope}. By inspection, $(\mathcal{M}+\mathcal{A}_\alpha)^{-1}$ is Lipschitz. Therefore, by \cref{remark:samesolution}, the sequence $\{(\mathcal{M}+\mathcal{A}_\alpha)^{-1}\mathcal{C}(1+\alpha)\widetilde{w}^k\}_k$ weakly converges to a fixed point of $\mathcal{T}_\alpha$. This gives us in particular the weak convergence of $\{x_i^k\}_k$, for all $i\in \{0,\dots,N\}$, to the same solution of \cref{sumDY}.
\end{proof}
As observed for the case $N=1$ in \cite[Theorem 2.1]{Davis2017}, the convergence of $\{x_i^k\}_k$ is strong if at least one operator is \textit{uniformly monotone} (cf. \cite[Definition 22.1 (iii)]{BCombettes}).
\begin{corollary}[Strong convergence]\label{prop:strongconvergenceParallelFDR}
	Assume that there exists $i \in \{0,\dots, N\}$ such that $A_i$ is uniformly monotone on every bounded set of $\dom A_i$, then for all $i \in \{0,\dots, N\}$ the sequences $\{x_i^k\}_{k}$ generated by \cref{parallelDY} converge strongly to a solution $x^*$ of \cref{sumDY}.
\end{corollary}
\begin{proof}
	Let $u^*$ be in $\zer \mathcal{A}$. Consider the bounded set $S = \{x^*\} \cup \{x_i^k\}_k \subset \dom A_i$. Then, by definition of uniform monotonicity, there exists an increasing function $\phi~:~\mathbb{R}_+ \to [0,+\infty]$ that vanishes only at $0$ such that
	\[ \langle p-p',x-x' \rangle \geq \phi\left(\|x-x'\|\right) \quad \textit{ for all } \ a_i\in A_ix,~ a_i'\in A_ix'.\]
	By definition of $\mathcal{T}_\alpha$ we have $ \mathcal{M}(u^k-\mathcal{T}_\alpha u^k) \in \mathcal{A}_\alpha\mathcal{T}_\alpha u^k $, thus
	\[\langle \mathcal{M}(u^k-\mathcal{T}_\alpha u^k),\mathcal{T}_\alpha u^k-u^*\rangle \geq \sigma \langle a_i^k-a_i^*,x_i^k-x^* \rangle \geq \phi\left(\|x_i^k-x^*\|\right), \]
	with $a_i^k\in A_i x_i^k$, $a_i^*\in A_i x^*$. Using \cref{ConvergenzaAu} and the fact that $\mathcal{T}_\alpha u^k\rightharpoonup u^*$ weakly in $\bH$, the left hand-side vanishes as $k\to +\infty$, thus $x_i^k \to x$ strongly in $\bH$. For all other sequences $\{x_i^k\}_k$, by inspection, we notice that \cref{ConvergenzaAu} also implies that $x_0^k-x_{i}^k \to 0$ for all $i \in \{1,\dots, N\}$. The thesis follows.
\end{proof}

\subsubsection{Sequential generalization of \textsf{FDR}}
The flexibility of the proposed framework allows to easily design new schemes for the $(2N+1)$-operator case that do not rely on the usual product space trick. In particular, we show how to retrieve a generalization of the \textsf{FDR} splitting that, in contrast to the iteration discussed in \cref{ParallelDavisYinSplitting}, admits an intrinsic sequential nature. Of course, setting $C_i=0$ for all $i$, this can be thought as a sequential generalization of \textsf{DRS} as well. Consider the following operators
(confer $\mathcal{C}^*$ below for a precise definition of $\mathcal{M} = \mathcal{C}\mathcal{C}^*$)

\vspace{-0.4cm}
{\footnotesize
	\begin{equation*}\label{SeqDavisYin}
	\arraycolsep=0pt
	\renewcommand{\arraystretch}{0.8}
	\mathcal{A}_{\alpha}=
	\begin{bmatrix}
	\alpha I+\sigma A_0 & -I & -I & & & \\
	I &  & -I & & & \\
	(1-2\alpha) I+\sigma C_1 & I & 2\alpha I+\sigma A_1 \rlap{$\ \ \ \cdots$} & & & \\
	& & \vdots \rlap{\hspace{0.6cm}$\ddots$} & \vdots & &\\
	& & & \llap{$\cdots \ \ \ $} 2\alpha I+\sigma A_{N-1} & -I & -I\\
	& & & I &  & -I \\
	& & & (1-2\alpha)I+\sigma C_N & I & \alpha I+\sigma A_N
	\end{bmatrix},~
	\arraycolsep=1.5pt
	\mathcal{M}=
    \begin{bmatrix}
	I & I & I & & & &\\
	I & I & I & & & &\\
	I & I & 2I & \cdots & & &\\
	& & \vdots & \ddots & \vdots & &\\
	& & & \cdots & 2I & I & I\\
	& & & & I & I & I\\
	& & & & I & I & I
    \end{bmatrix}.
  \end{equation*}%
}%
Also in this case, if $u=(x_0,v_1,x_1,\dots,v_N,x_N) \in \bH:= \mathcal{H}^{2N+1}$ satisfies $0\in\mathcal{A}_\alpha u$ then $x =x_0=\dots =x_N$ solves \cref{sumDY} and a solution $x$ of \cref{sumDY} yields $0\in \mathcal{A}_\alpha u$ setting $x_0=\dots=x_N=x$ and $v_i=(\alpha-1)x-a_i-\sigma C_i x$, for all $i\in\{1,\dots,N\}$, where $a_i\in \sigma A_i x$, $i\in\{0,\dots,N\}$, such that $0=a_0+\sum_{i=1}^N a_i + \sigma C_i x$. The operator $(\mathcal{M}+\mathcal{A}_\alpha)$ is lower-triangular and easy to invert, leading to a Lipschitz continuous inverse. We choose the decomposition of $\mathcal{M}$ (according to \cref{CholeskyDecomposition}) to be $\mathcal{M}=\mathcal{C}\mathcal{C}^*$ with $\mathcal{C}^*:u\mapsto (x_0+v_1+x_1,~\dots~,~x_{N-1}+v_N+x_N)$ giving the reduced variables $w_i = x_{i-1}+v_i+x_i$ for all $i\in\{1,\dots,N\}$, which for the sake of exposition we re-scale again with a factor of $(1+\alpha)^{-1}$, considering $\widetilde{w}_i=w_i/(1+\alpha)$ instead of $w_i$ for $i\in\{1,\dots,N\}$. We put $\gamma = \sigma/(1+\alpha)$. So that, even if we introduced $(2N+1)$ variables, the reduced algorithm would only need to store $N$. Again, the operator $\widetilde{\mathcal{T}}_\alpha=\left(I+\mathcal{C}^*\rhd \mathcal{A}_\alpha\right)^{-1}$ can be computed explicitly and the relaxed reduced algorithm reads as
\begin{equation}\label{eq:SequentialFDR}
\left\{
\begin{aligned}
x_0^{k+1}&=J_{\gamma A_0}\widetilde{w}_1^{k},\\
x_i^{k+1} &=J_{\frac{\gamma}{2}A_i}\bigg(x_{i-1}^{k+1}+\frac{\widetilde{w}_{i+1}^{k}-\widetilde{w}_{i}^{k}}{2}-\frac{\gamma}{2}C_{i}x_{i-1}^{k+1}\bigg) & \textit{ for } i \in \{1,\dots, N-1\},\\
x_N^{k+1}&=J_{\gamma A_N}\left(2x_{N-1}^{k+1}-\widetilde{w}_{N}^k-\gamma C_{N}x_{N-1}^{k+1}\right),\\
\widetilde{w}_i^{k+1}&= \widetilde{w}_i^k+\theta_k(x_{i}^{k+1}-x_{i-1}^{k+1}) & \textit{ for } i \in \{1,\dots, N\}.
\end{aligned}
\right.
\end{equation}

\begin{theorem}\label{thm:SeqDYS}
	Let $\widetilde{w}^k=(\widetilde{w}_1^k,...,\widetilde{w}_N^k)$ for $k \in \mathbb{N}$ be the sequence generated by the sequential \textsf{FDR} scheme \cref{eq:SequentialFDR} with starting point $\widetilde{w}^0\in \mathcal{H}^N$. Let $\gamma\in(0,4\beta)$ and $\{\theta_k\}_k$ satisfy $\theta_k\in\big[0,2-\frac{\gamma\beta}{2}\big]$ and $\sum_k\theta_k\big(2-\frac{\gamma\beta}{2}-\theta_k\big)=+\infty$. If a solution to \cref{sumDY} exists, then $\{\widetilde{w}^k\}_k$ converges weakly to a $\widetilde{w}^*\in \mathcal{H}^N$ such that $x^*:=J_{\gamma A_0}\widetilde{w}_1^*$ is a solution to \cref{sumDY}. Moreover, the sequences $\{x_i^k\}_k$, for $i \in \{0,\dots,N\}$, weakly converge to this solution.
\end{theorem}

The proof is similar to \cref{thm:parallelDY} and therefore omitted.

\begin{corollary}[Strong convergence]
	If there exists $i \in \{0,\dots, N\}$ such that $A_i$ is uniformly monotone on every bounded set of $\dom A_i$, then for all $i \in \{0,\dots, N\}$ the sequences $\{x_i^k\}_{k}$ generated by \cref{eq:SequentialFDR} converge strongly to a solution $x^*$ of \cref{sumDY}.
\end{corollary}
Again, the proof follows the lines of the parallel case, see \cref{prop:strongconvergenceParallelFDR}. 

\subsection{Numerical experiments}\label{sec:Experiment}

In this section, we show the numerical performance of the algorithms discussed in \cref{sec:applications}, in particular, we will focus on the novel sequential generalization of \textsf{FDR} and its parallel counter-variant. We do not claim to out-perform state-of-art optimization methods. We only show that indeed, the sequential \textsf{FDR} reaches comparable performance when the computations are performed on a single \textsf{CPU} and the number of addends is small. For all the methods, the parameters were roughly tuned for best performance. 

\textit{Problem formulation.} A classical benchmark for testing optimization algorithms, already considered, for instance, in \cite{Ryu}, is the so-called \emph{Markowitz portfolio optimization problem}. Taking into account transaction costs we have a problem of the form
\begin{equation*}
\min_{w \in \Delta} \ w^T \Sigma w - r^T w +\frac{\delta}{2}\|w\|^2 + \sum_{i=1}^n|w_i-(w_0)_i| + \sum_{i=1}^n |w_i-(w_0)_i|^{3/2},
\end{equation*}
where $\delta > 0$, $r \in \mathbb{R}^n$ is a vector of estimated assets returns, $\Sigma$ is the estimated covariance matrix of returns (which is a symmetric positive semi-definite $n\times n$ matrix), $\Delta\subset \mathbb{R}^n$ is the standard simplex and $w_0 \in \mathbb{R}^n$ is the initial position. The term $\phi(w)= \sum_{i=1}^n|w_i-(w_0)_i|  + \sum_{i=1}^n |w_i-(w_0)_i|^{3/2}$ is a standard penalization for financial transactions \cite{BoydPortfolio} and $(\delta/2)\|w\|^2$ is motivated from a robust-optimization viewpoint \cite{Ho2015WeightedEN}. 

\textit{Optimization procedure.} We apply the sequential and the parallel generalizations of \textsf{FDR} to this problem splitting the objective function into the following terms
\begin{equation*}
\min_{w \in \mathbb{R}^n} \ \underbrace{w^T \Sigma w - r^T w +\frac{\delta}{2}\|w\|^2}_{f(w)} + \underbrace{\sum_{i=1}^n|w_i-(w_0)_i|}_{g_0(w)} + \underbrace{\sum_{i=1}^n |w_i-(w_0)_i|^{3/2} }_{g_1(w)} +\underbrace{\mathbb{I}_\Delta (w)}_{g_2(w)}.
\end{equation*}
For the sake of completeness we will also split the regular term $f$ into the two parts $f_1(w)=w^T \Sigma w - r^T w$ and $f_2(w)=(\delta/2)\|w\|^2$. We denote by $L$ the largest eigenvalue of $\Sigma$, i.e.,~the Lipschitz constant of $w \mapsto \Sigma w$. We consider the following three versions of sequential \textsf{FDR} applied to $0\in (A_0+A_1+A_2+C_1+C_2)x $ with $A_i=\partial g_i$ for $i=0,1,2$
\begin{enumerate}
	\item  \textsf{SeqFDRv1}: we split the regular term, meaning that $C_1=\nabla f_1$ and $C_2=\nabla f_2$. Thus, we may set $\gamma = \min\{1/L, 1/\delta \}$ and $\theta_k = 1$ for all $k \in \mathbb{N}$.
	\item \textsf{SeqFDRv2}: we do not split the regular term but consider $C_1 = \nabla f$ and $C_2=0$. Thus, we may set $\gamma = 1/(L+\delta)$  and $\theta_k = 1$ for all $k \in \mathbb{N}$.
	\item \textsf{SeqFDRv3}: we repeat two times the regular term, i.e.,~$C_1 = (1/2)\nabla f$ and $C_2=(1/2)\nabla f$. Thus, we may set $\gamma = 2/(L+\delta)$ and $\theta_k = 1$ for all $k \in \mathbb{N}$. 
\end{enumerate}

\smallskip

Their convergence behaviour, in particular the differences, are shown in Figure \ref{fig:PortOpt}. We select the best version, in this case \textsf{SeqFDRv3}, and we compare its behaviour with the following algorithms. The generalized Backward-Forward (\textsf{GenBF}) in \cite{raguet:GenFB}, which is actually an instance of the parallel \textsf{FDR} applied to a larger problem with four non-regular terms: $\widetilde{A_0}=0$ and $\widetilde{A_1}=A_0, \ \widetilde{A_2}=A_1, \ \widetilde{A_3}=A_2$, hence ending up with three variables. The parallel \textsf{FDR} with non-split regular term (\textsf{ParFDR}), with the parallel \textsf{DR} (\textsf{ParDR}) introducing $g_3=f$ and computing its proximal operator (having thus one additional variable), and eventually with \textsf{PPXA} \cite{PPXA}, which, as \textsf{ParDR}, requires three variables. Since our algorithms and \textsf{ParFDR} have the lesser number of variables to store and the costs per iterations of all methods are similar, we show the performance only in terms of iterations.

\begin{figure}[t!]
	\centering
	\includegraphics[width=1\linewidth]{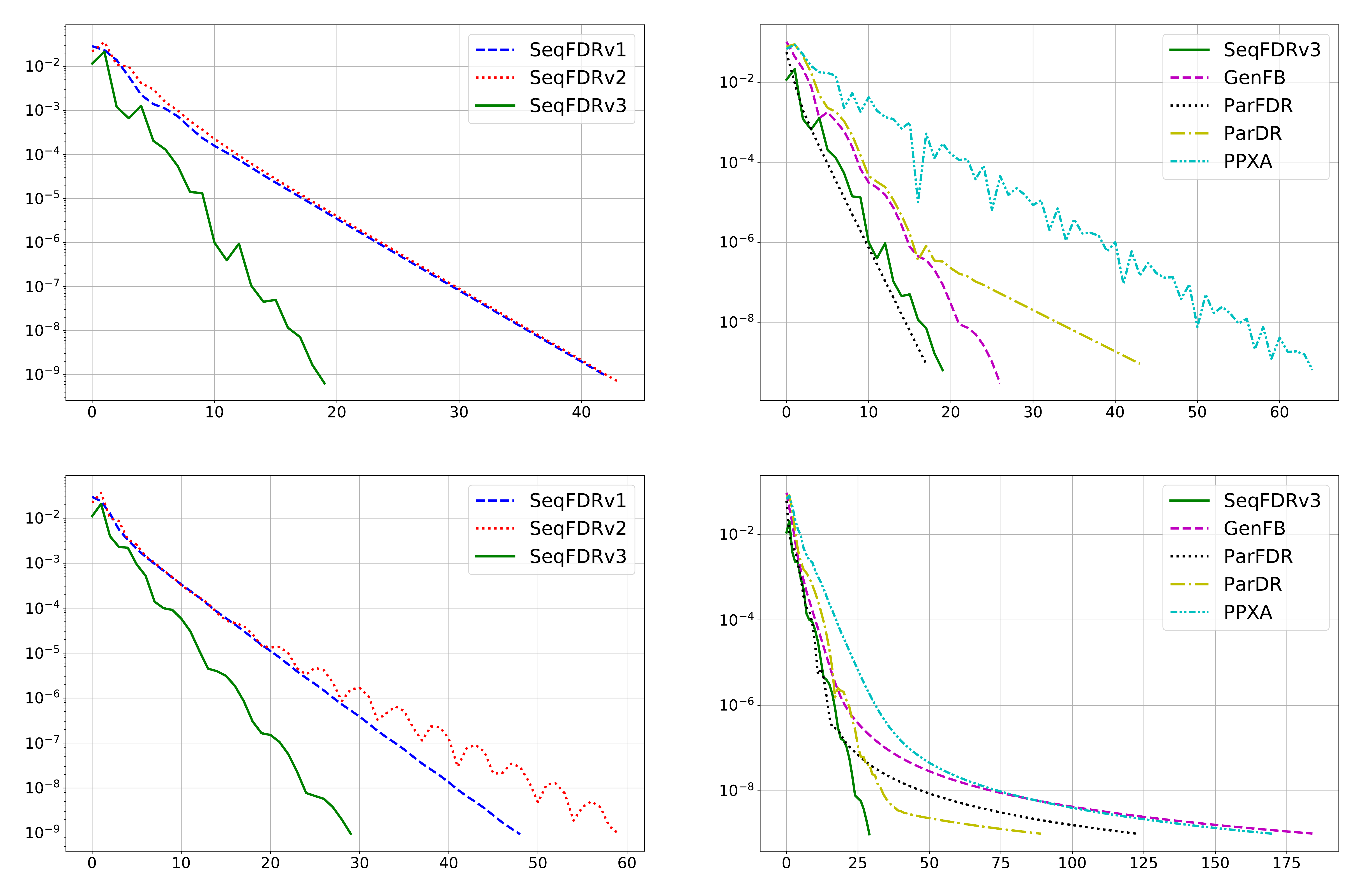}
	\caption{Norm distance to the minimizer as a function of the number of iterations. The estimated returns $r$ and the matrix $\Sigma$ are derived from real data\protect\footnote{Blah} considering 200 market days and $n=53$ assets. We show two cases: (above) the initial state $w_0$ has been chosen randomly and (below) $w_0$ is the output of the first one considered on the same optimization problem but 20 days later.}
	\label{fig:PortOpt}
\end{figure}
\footnotetext{Downloaded from \url{https://stanford.edu/class/engr108/portfolio.html} (accessed on March 27, 2021).}

\paragraph{Comments}  Figure \ref{fig:PortOpt} shows that \textsf{SeqFDRv3} performs as well as the other algorithms and reaches high precision within a  few iterations. In the literature, there are no essential guidelines in splitting the forward term, indeed, it is not even considered as a possibility in \cite{raguet:GenFB, condat:hal-00609728}. However, we believe that is of interest when the computation of the full gradient is computationally expensive, so that splitting it into two parts or computing it only once (instead of $N$ times) at each iteration would result in more efficient updates. During the experiments, since the projections onto affine sets are much easier to compute than the projection onto the simplex which in practice takes $\mathcal{O}(n)$ time, we also split the simplex constraint into its affine components $\mathbb{I}_\Delta(w) = \mathbb{I}\{w: \ w_1+\dots + w_n=1\}+ \mathbb{I}\{w : \ w_i \geq 0  \ \text{for all} \ i \}$ at the cost of one additional variable. However, in our small-scale problem, splitting the simplex constraint behaves quite badly and therefore, we omitted it.

\section{Conclusions}

In this paper, we studied preconditioned proximal point algorithms with a particular focus on the degenerate case. We established weak and strong convergence results and proposed a new perspective on well-known algorithms that directly yields sharp conditions for the parameters. Moreover, an intuitive way to generalize those algorithms to tackle the sum of many operators is shown. In the future, we plan to: (i) better understand the limitations we found for the parameters' freedom regarding the convergence of $\{u^k\}_k$ in \cref{WeakConvergencePPPnew}; (ii) study a full characterisation \textit{à la Minty} of the requirements we are implicitly imposing on $\mathcal{A}$ with assumptions \cref{HpAdmissiblePrec}; (iii) analyze overrelaxed methods, such as the overrelaxed Forward Peaceman-Rachford method mentioned in \cref{rem:overrelaxedFDR}, and their sequential and parallel extensions.

\section*{Appendix}

\begin{proof}[Proof of \cref{CholeskyDecomposition}]
	In finite dimension, i.e.,~for $\bH=\mathbb{R}^d$, one shows via spectral decomposition, that every linear positive semi-definite operator admits such a decomposition, and, moreover, $\bD$ can be chosen as $\mathbb{R}^{d-k}$, where $k=\dim \left(\ker \mathcal{M}\right)$. One way to construct such a decomposition in infinite dimension is as follows: since $\mathcal{M}$ is self-adjoint positive semi-definite then $\mathcal{M}$ admits an (actually unique) self-adjoint square root (see~\cite[Chapter VII]{riesz2012functional}, or  \cite[Theorem 4]{bernau_1968}) which means that there exists a bounded self-adjoint operator $\mathcal{E}:\bH\to\bH$ such that $\mathcal{M}=\mathcal{E}^2$. Let us denote by $\bD=\overline{\Image\mathcal{E}}$, which is a real Hilbert space. Denoted by $i:\bD \to \bH$ the inclusion operator, define $\mathcal{C}:=\mathcal{E} i$. We have that $\mathcal{C}$ is injective, in fact
	\begin{equation*}
		\ker \mathcal{C}=\bD\cap \ker \mathcal{E}=\overline{\Image\mathcal{E}}\cap \ker \mathcal{E}=\left(\ker \mathcal{E}\right)^\perp\cap \ker \mathcal{E}=\{0\}.
	\end{equation*}
	Moreover, since $i i^*d=d$ for all $d \in \bD\subset \bH$, it follows that $\mathcal{C}\mathcal{C}^*=\mathcal{E} i i^* \mathcal{E}^*=\mathcal{E}\mathcal{E}^*=\mathcal{E}\mathcal{E}=\mathcal{M}$. Notice that the image of $\mathcal{E}$ is dense in $\bD$, hence, since $i^*d=d$ for all $d \in \bD$, we have $\Image\mathcal{E}=\Image\mathcal{C}^*$ and the image of $\mathcal{C}^*$ is dense in $\bD$ and $\mathcal{C}^*$ would be onto if it has closed range. We prove by contradiction that this is the case when $\mathcal{M}$ has closed range. Let $\{x^k\}_k$ be a sequence in $\bH$ such that $\mathcal{C}^*x^k \to y$ with $y\in \bD \setminus \Image\mathcal{C}^*$, then $\mathcal{M}x^k=\mathcal{C}\mathcal{C}^*x^k\to \mathcal{C}y$. Since $\mathcal{M}$ has closed range there exists an $x\in \bH$ such that $\mathcal{C}y=\mathcal{M}x=\mathcal{C}\mathcal{C}^*x$ so that $\mathcal{C}(y-\mathcal{C}^*x)=0$ and $y=\mathcal{C}^*x$ by injectivity of $\mathcal{C}$, which yields a contradiction.
\end{proof}

\begin{proof}[Proof of \cref{prop:MMaximalMonotonicityProp}]
	We start by proving that if $\mathcal{M}^{-1}\mathcal{A}$ is $\mathcal{M}$-monotone, then $\mathcal{A}\cap (\bH \times \Image\mathcal{M})$ is monotone. Let $(u,v), \ (u',v')\in \mathcal{A}\cap (\bH \times \Image\mathcal{M})$, then by definition there exist $\widetilde{v}, \ \widetilde{v}'$ such that $(u,v)=(u, \mathcal{M}\widetilde{v})$ and $(u',v')=(u', \mathcal{M}\widetilde{v}')$, thus, since $(u,\widetilde{v}), \ (u',\widetilde{v}') \in \mathcal{M}^{-1}\mathcal{A}$
	\begin{equation*}
		\langle v-v',u-u'\rangle = \langle\mathcal{M}\widetilde{v}-\mathcal{M}\widetilde{v}',u-u'\rangle = \langle \widetilde{v}-\widetilde{v}',u-u'\rangle_\mathcal{M} \geq 0,
	\end{equation*}
	by the $\mathcal{M}$-monotonicity of $\mathcal{M}^{-1}\mathcal{A}$. Conversely, given $(u,\widetilde{v}), \ (u',\widetilde{v}') \in \mathcal{M}^{-1}\mathcal{A}$, since $(u,v):=(u, \mathcal{M}\widetilde{v})$ and $(u',v'):=(u', \mathcal{M}\widetilde{v}')$ belong to $\mathcal{A}\cap (\bH \times \Image\mathcal{M})$ we have
	\begin{align*}
		\langle \widetilde{v}-\widetilde{v}',u-u'\rangle_\mathcal{M}  = \langle\mathcal{M}\widetilde{v}-\mathcal{M}\widetilde{v}',u-u'\rangle =\langle v-v',u-u'\rangle \geq 0.
	\end{align*}
\end{proof}

\begin{proof}[Proof of \cref{thm:DRS-strongconvergence}]
	Our goal is to prove that there exists a constant $\alpha>0$ such that \cref{eq:AtildeCocoercive2} holds. Let $u=(x,y), \ u'=(x',y')$ and  $(u,\xi),(u',\xi')\in\mathcal{A}\cap( \bH\times\Image \mathcal{M})$, then 
	\begin{equation*}
		\xi\in \mathcal{A}u=\begin{pmatrix}
			\sigma Ax+y\\
			-x+(\sigma B)^{-1}y
		\end{pmatrix},\quad \xi'\in \mathcal{A}u'=\begin{pmatrix}
			\sigma Ax'+y'\\
			-x'+(\sigma B)^{-1}y'
		\end{pmatrix}.
	\end{equation*}
	This means that there exists $a\in \sigma Ax, \ a'\in \sigma Ax'$ and $b\in(\sigma B)^{-1}y, b'\in(\sigma B)^{-1}y'$ such that $\xi_1=a+y, \ \xi_2=-x+b$ and $\xi_1'=a'+y',\ \xi_2'=-x'+b'$. Moreover, since $\xi=(\xi_1,\xi_2),~\xi'=(\xi_1',\xi_2') \in \Image\mathcal{M}$, we have $\xi_1=-\xi_2$ and $\xi_1'=-\xi_2'$. Together, $a+y=x-b$, $a'+y'=x'-b'$, such that
	\begin{equation}\label{eq:SostituzioneDRperDelta}
		\Delta a + \Delta y=\Delta x - \Delta b,       
	\end{equation}
	where $\Delta a=a-a', \ \Delta b=b-b', \ \Delta x=x-x'$ and $\Delta y = y-y'$.
	
	Now we prove all claimed cases by checking the strong $\mathcal{M}$-monotonicity:
	
	\noindent\textbf{Case 1a)} $A$ $\mu$-strongly monotone and $B$ $1/\beta$-cocoercive. Note that $\sigma A$ is $\sigma \mu$-strongly monotone and $(\sigma B)^{-1}$ is $\frac{1}{\sigma \beta}$-strongly monotone. With the notations introduced above we have that
	\begin{align*}
		\langle \xi-\xi', u-u'\rangle & =\langle \Delta a +\Delta y, \Delta x\rangle+\langle -\Delta x + \Delta b, \Delta y\rangle= \langle \Delta a,\Delta x\rangle + \langle \Delta b,\Delta y\rangle\\
		&\geq \sigma \mu \|\Delta x\|^2 + \frac{1}{\sigma \beta}\|\Delta y\|^2 \geq \min\bigg\{\sigma \mu,\frac{1}{\sigma \beta}\bigg\}(\|\Delta x\|^2+\|\Delta y\|^2)\\ &\geq \min\bigg\{\frac{\sigma\mu}{2},\frac{1}{2\sigma \beta}\bigg\}\|\Delta x-\Delta y\|^2=\alpha \|C^*\Delta u\|^2=\alpha\|\Delta u\|^2_{\mathcal{M}},
	\end{align*}
	where $\Delta u= u-u'$. This establishes \cref{eq:AtildeCocoercive2} and gives us a convergence rate of $1/(1+\alpha)$ with $\alpha=\min\{\frac{\sigma \mu}{2},\frac{1}{2\sigma \beta}\}$.         
	
	\noindent \textbf{Case 1b)} $B$ $\mu$-strongly monotone and $A$ $1/\beta$-cocoercive. Here, $\sigma A$ is  $\frac{1}{\sigma \beta}$-strongly monotone and $(\sigma B)^{-1}$ is $\sigma \mu$-strongly monotone. We have
	\begin{align*}
		\langle \xi-\xi', u-u'\rangle & = \langle \Delta a,\Delta x\rangle + \langle \Delta b,\Delta y\rangle\\
		&\geq \frac{1}{\sigma \beta} \|\Delta a\|^2 + \sigma \mu \|\Delta b\|^2 \geq \min\biggl\{\sigma \mu,\frac{1}{\sigma \beta}\biggr\}(\|\Delta a\|^2+\|\Delta b\|^2)\\ &\geq \min\bigg\{\frac{\sigma \mu}{2},\frac{1}{2\sigma \beta}\bigg\}\|\Delta a+\Delta b\|^2.
	\end{align*}
	Finally, making use of \cref{eq:SostituzioneDRperDelta}, we obtain \[ %\langle \Delta \xi, \Delta u\rangle
	\langle  \xi-\xi', u-u'\rangle
	\geq \min\bigg\{\frac{\sigma \mu}{2},\frac{1}{2\sigma \beta}\bigg\}\|\Delta x-\Delta y\|^2 =\alpha \|C^*\Delta u\|^2=\alpha\|\Delta u\|^2_{\mathcal{M}}. \]
	
	\noindent \textbf{Case 2a)} $A$ $\mu$-strongly monotone and $1/\beta$-cocoercive. Using monotonicity of $B$, $\sigma \mu$-strong monotonicity, $\frac{1}{\sigma \mu}$-cocoercivity of $A$ and \cref{eq:SostituzioneDRperDelta}, we have, introducing $t\in (0,1)$, that
	\begin{align*}
		\langle & \xi-\xi', u-u'\rangle=\langle \Delta a,\Delta x\rangle +\langle \Delta b,\Delta y\rangle \geq (1-t)\langle \Delta a,\Delta x\rangle+t\langle \Delta a,\Delta x\rangle\\
		&\geq (1-t)\sigma\mu \|\Delta x\|^2+\frac{t}{\sigma\beta}\|\Delta a\|^2 = (1-t)\sigma\mu \| \Delta x \|^2+\frac{t}{\sigma\beta}\|\Delta x-\Delta y - \Delta b\|^2\\
		&=(1-t)\sigma\mu \|\Delta x \|^2+\frac{t}{\sigma \beta} \|\Delta x-\Delta y\|^2+\frac{t}{\sigma \beta} \|\Delta b\|^2-\frac{2t}{\sigma \beta} \langle \Delta x-\Delta y, \Delta b\rangle\\
		&=(1-t)\sigma\mu \|\Delta x \|^2-\frac{2t}{\sigma \beta} \langle \Delta x, \Delta b\rangle +\frac{t}{\sigma \beta} \|\Delta b\|^2+\frac{t}{\sigma \beta} \|\Delta x-\Delta y\|^2+\frac{2t}{\sigma \beta} \langle \Delta y, \Delta b\rangle.
	\end{align*}
	Choosing $t:=(1+(\sigma^2\mu\beta)^{-1})^{-1}\in (0,1)$ we see that $(1-t)\sigma\mu=t/(\sigma\beta)$. Thus, the first three terms of the latter expression together yield a non-negative term. Therefore, using that $\langle \Delta b, \Delta y\rangle\geq 0$ we get
	\[
	\langle \xi-\xi',u-u'\rangle \geq \frac{t}{\sigma \beta} \|\Delta x -\Delta y\|^2 =\frac{t}{\sigma \beta} \|u-u'\|_\mathcal{M}^2=\frac{\sigma \mu}{\sigma^2\mu\beta+1}\|u-u'\|_\mathcal{M}^2.
	\]
	\noindent \textbf{Case 2b)} $B$ $\mu$-strongly monotone and $1/\beta$-cocoercive. Using monotonicity of $A$, $\sigma \mu$-strong monotonicity, $\frac{1}{\sigma \mu}$-cocoercivity of $B$ and \cref{eq:SostituzioneDRperDelta}, we have, introducing $t\in (0,1)$, that
	\begin{align*}
		&\langle \xi-\xi', u- u'\rangle=\langle \Delta a,\Delta x\rangle+\langle \Delta b,\Delta y\rangle\geq \langle \Delta a,\Delta x\rangle+ (1-t)\langle \Delta b,\Delta y\rangle+t\langle \Delta b,\Delta y\rangle\\
		&\geq \langle \Delta a,\Delta x\rangle+\frac{1-t}{\sigma \beta} \|\Delta y\|^2+t\sigma \mu\|\Delta b\|^2\\
		&=\langle \Delta a,\Delta x\rangle+\frac{1-t}{\sigma \beta} \|\Delta y\|^2+t\sigma \mu\|\Delta x -\Delta a-\Delta y\|^2\\
		&=\langle \Delta a,\Delta x\rangle+\frac{1-t}{\sigma \beta} \|\Delta y\|^2+t\sigma \mu \|\Delta x-\Delta y\|^2+t\sigma \mu \|\Delta a\|^2-2t\sigma \mu \langle \Delta x-\Delta y, \Delta a\rangle\\	&=\frac{1-t}{\sigma \beta} \|\Delta y\|^2+2t\sigma \mu \langle \Delta y, \Delta a\rangle+t\sigma \mu \|\Delta a\|^2+\text{{\footnotesize $(1-2t\sigma \mu)$}} \langle \Delta a, \Delta x\rangle+t\sigma \mu \|\Delta x-\Delta y\|^2.
	\end{align*}
	Choosing $t:=(1+\sigma^2\mu\beta)^{-1}\in (0,1)$ yields $t\sigma\mu=(1-t)/(\sigma\beta)$. Thus, the first three terms of the latter expression together yield a non-negative term. Moreover, using the obvious relation $\mu\leq \beta$ we have $(1-\frac{2\sigma\mu}{1+\sigma^2\mu\beta})\geq 0$, so we can use $\langle \Delta a, \Delta x\rangle\geq 0$ to obtain
	\begin{align*}
		\langle \xi-\xi',u-u'\rangle&\geq t\sigma \mu \|\Delta x -\Delta y\|^2 =t\sigma \mu \|u-u'\|_\mathcal{M}^2=\frac{\sigma \mu}{\sigma^2\mu\beta+1}\|u-u'\|_\mathcal{M}^2.
	\end{align*}
	
	\noindent \textbf{Case 3)} $A$ $\mu$-strongly monotone and $\beta$-Lipschitz continuous or vice-versa. It is easy to see that if an operator is $\mu$-strongly monotone and $\beta$-Lipschitz then it is also $\mu/\beta^2$ cocoercive. Therefore, we can apply case 2 and obtain the rate $1/(1+\alpha)$ where $\alpha=\frac{\sigma\mu}{\sigma^2\beta^2+1}$.
\end{proof}

\bibliographystyle{siamplain}
\bibliography{references}
\end{document}